\documentclass[12pt,a4paper]{amsart}
\usepackage{a4wide}
\usepackage{amsmath, amssymb, amsfonts,enumerate}
\usepackage[all]{xy}
\usepackage{amscd}
\usepackage{comment}
\usepackage{mathtools}
\usepackage{enumitem}
\usepackage{tikz} 
\usepackage{url}

\newtheorem{theorem}{Theorem}[section]
\newtheorem{lemma}[theorem]{Lemma}
\newtheorem{corollary}[theorem]{Corollary}
\newtheorem{proposition}[theorem]{Proposition}

 \theoremstyle{definition}
 \newtheorem{definition}[theorem]{Definition}
 \newtheorem{remark}[theorem]{Remark}

 \newtheorem{example}[theorem]{Example}

\numberwithin{equation}{section}
\newcommand {\N}{\mathbb{N}} 
\newcommand {\Z}{\mathbb{Z}} 
\newcommand {\R}{\mathbb{R}} 
\newcommand {\Q}{\mathbb{Q}} 
\newcommand {\C}{\mathbb{C}} 
\newcommand{\F}{\mathbb{F}}

\newcommand{\DD}{\mathcal{D}}

\newcommand{\FF}{\mathcal{F}}

\newcommand{\OO}{\mathcal{O}}
\newcommand{\PP}{\mathcal{P}}

\newcommand{\A}{\mathbb{A}} 

\newcommand{\Sch}{\mathrm{Sch}}
\newcommand{\Proj}{\mathbb{P}}

   \DeclareMathOperator{\GL}{GL}

\DeclareMathOperator{\Fix}{Fix}

\DeclareMathOperator{\CA}{CA}

\DeclareMathOperator{\Mor}{Mor}

\DeclareMathOperator{\Id}{Id}

\DeclareMathOperator{\Spec}{Spec}

\begin{document}
\title{On injective endomorphisms of symbolic schemes}
\author[T.Ceccherini-Silberstein]{Tullio Ceccherini-Silberstein}
\address{Dipartimento di Ingegneria, Universit\`a del Sannio, C.so
Garibaldi 107, 82100 Benevento, Italy}
\email{tullio.cs@sbai.uniroma1.it}
\author[M.Coornaert]{Michel Coornaert}
\address{Universit\'e de Strasbourg, CNRS, IRMA UMR 7501, F-67000 Strasbourg, France}
\email{michel.coornaert@math.unistra.fr}
\author[X.K.Phung]{Xuan Kien Phung}
\address{Universit\'e de Strasbourg, CNRS, IRMA UMR 7501, F-67000 Strasbourg, France}
\email{phung@math.unistra.fr}
\subjclass[2010]{37B15, 14A10, 14A15, 37B10,       68Q80}
\keywords{cellular automaton, scheme, algebraic variety, closed image property, surjunctivity, reversibility, invertibility}
\begin{abstract}
Building on the seminal work of Gromov on endomorphisms of symbolic algebraic varieties~\cite{gromov-esav},
we introduce a notion of cellular automata over schemes which generalize affine algebraic cellular automata in \cite{cc-algebraic-ca}. 
We extend known results to this more general setting. We also establish several new ones regarding the closed image property, surjunctivity, reversibility, and invertibility for cellular automata over algebraic varieties with coefficients in an algebraically closed field.
As a byproduct, we obtain a negative answer to a question raised in~\cite{cc-algebraic-ca} on the existence of a bijective complex affine algebraic cellular automaton $\tau \colon A^\Z \to A^\Z$ whose inverse is not algebraic. 
\end{abstract}
\date{\today}
\maketitle

\setcounter{tocdepth}{1}
\tableofcontents

\section{Introduction}

The goal of the present paper is to investigate certain properties of cellular automata over schemes. 
In the setting and terminology of Gromov~\cite{gromov-esav}, these cellular automata are the
``endomorphisms of symbolic schemes'', that is, the proregular maps of projective limits over a directed system of schemes. 
However, here we use the standard notation and terminology coming from symbolic dynamics, which is well adapted to our investigations
and techniques, see, e.g.,~\cite{ca-and-groups-springer}.
We introduce a very general definition for such cellular automata (see Definition~\ref{def:ca-over-schemes} below).
The ambient universe is always a group and the local defining rules are assumed to be invariant under left multiplication and induced by scheme morphisms.  
Under certain hypotheses on the group and the schemes involved,
we obtain surjunctivity, reversibility, and invertibility  results for the corresponding cellular automata. 
\par 
Let us first recall some basic notions about cellular automata and introduce notation. 
\par
Let  $A$ be a set, called the \emph{alphabet},  and let $G$ be a group, called the \emph{universe}.
The set $A^G = \{c \colon G \to A\}$, consisting of all maps from $G$ to $A$,  is called the set of \emph{configurations} over the group $G$ and the alphabet $A$.
 The action of $G$ on itself by left multiplication induces an action of $G$ on $A^G$ defined by
 $(g,c) \mapsto gc$, where
\[
(g c)(h) \coloneqq  c(g^{-1}h) \qquad   \text{for all  }  g,h \in G \text{  and  } c \in A^G.
\]
This action is called the $G$-\emph{shift} on $A^G$.
\par
Given a configuration $c \in A^G$ and a subset $\Omega \subset G$, the element $c\vert_\Omega \in A^\Omega$ defined by $c\vert_\Omega(g) = c(g)$ for all $g \in \Omega$ is called the 
\emph{restriction} of the configuration  $c$ to $\Omega$.  
\par
A \emph{cellular automaton} over the group $G$ and the alphabet $A$ is a map 
$\tau \colon A^G \to A^G$ satisfying the following condition: 
there exist a finite subset $M \subset G$ 
and a map $\mu_M \colon  A^M \to A$ such that 
\begin{equation} 
\label{e;local-property}
(\tau(c))(g) = \mu_M((g^{-1} c )\vert_M)  \quad  \text{for all } c \in A^G \text{ and } g \in G
\end{equation}
(see e.g.~\cite{ca-and-groups-springer}).
 Such a set $M$ is then called a \emph{memory set} of $\tau$ and $\mu_M$ is called the \emph{local defining map}    associated with $M$.
  \par
The class of algebraic cellular automata was introduced and investigated in~\cite{cc-algebraic-ca}
after the seminal work of Gromov~\cite{gromov-esav} we alluded to above. 
Suppose that $K$ is a field and $V$ is an \emph{affine algebraic set}  over $K$, i.e., 
$V \subset K^n$ for some integer $n \geq 1$ and $V$ is the set of common zeroes  of a family of polynomials in $n$ variables with coefficients in $K$.  
A cellular automaton $\tau \colon V^G \to V^G$ with alphabet $V$ is called an \emph{algebraic cellular automaton} if it admits a memory set $M \subset G$ such that the associated local defining map
$\mu_M \colon V^M \to V$ is \emph{regular}, i.e., the restriction of some polynomial map $(K^n)^M \to K^n$.  
\par
Recall that if $X$ and $Y$ are schemes based over a common base scheme $S$, the set of $Y$-points of $X$
is the set $X(Y)$ consisting of all morphisms of $S$-schemes from $Y$ to $X$
(see Section~\ref{sec:schemes} for the relevant background material on schemes).

\begin{definition}
\label{def:ca-over-schemes}
Let $G$ be a group 
and let $S$ be a  scheme.
Let $X$ and $Y$ be schemes over $S$
and let $A \coloneqq X(Y)$ denote the set of $Y$-points of $X$. 
A \emph{cellular automaton over the group $G$ and the $S$-scheme $X$ with coefficients in the $S$-scheme} $Y$ 
is a cellular automaton
$\tau \colon A^G \to A^G$ over the group $G$ and the alphabet $A$  
that admits a memory set $M$ such that  the associated local defining map
$\mu_M  \colon A^M \to A$ is induced by some $S$-scheme morphism
$f \colon X^M \to X$, where $X^M$ denotes the $S$-fibered product of a  family of copies of $X$ indexed by $M$.
In that case, one also says that $\tau$ is a \emph{cellular automaton over the group $G$ and the schemes  $S,X,Y$}.
  \end{definition}

It is not difficult to check that this definition makes sense. 
More specifically, observe that  one has $X^M(Y) = (X(Y))^M = A^M$  by the universal property of $S$-fibered products, so that any $S$-scheme morphism $f \colon X^M \to X$  induces a map 
$f^{(Y)} \colon A^M \to A$ defined by $f^{(Y)}(\varphi) = f \circ \varphi$ for every $S$-scheme morphism 
$\varphi \colon Y \to X^M$.
\par
Moreover, Definition~\ref{def:ca-over-schemes}   generalizes that  of an algebraic cellular automaton given in~\cite{cc-algebraic-ca}.
Indeed, if $K$ is a field and $V \subset K^n$ an algebraic set,
 there is an $S$-scheme $X$ associated with $V$ for $S = \Spec(K)$,
namely  $X = \Spec(K[u_1,\dots,u_n]/I)$, where $I = I(V)$ is the ideal of $K[u_1,\dots,u_n]$ consisting of all polynomials that identically vanish on $V$.
One then has $A = X(S) = V$.
On the other hand, 
the regular maps between two regular sets $V_1 \subset K^{n_1}$ and $V_2 \subset K^{n_2}$ 
are precisely those induced by the $S$-morphisms between their corresponding $S$-schemes $X_1$ and $X_2$, equivalently,  
the $K$-algebra morphisms from
 $K[z_1,\dots,z_{n_2}]/I(V_2)$ to $K[t_1,\dots,t_{n_1}]/I(V_1)$.
 Thus, $\tau \colon V^G \to V^G$ is an algebraic cellular automaton,
as defined in~\cite{cc-algebraic-ca},  if and only if
 $\tau$ is a cellular automaton in the sense of Definition~\ref{def:ca-over-schemes} 
 over the schemes $S, X,  Y$
for  $S = Y = \Spec(K)$ and $X$ is the $S$-scheme associated with $V$.  
  \par
  Actually, any cellular automaton $\tau \colon A^G \to A^G$ can be regarded, in a trivial way, as a cellular automaton  over suitably chosen schemes.
Indeed, suppose that $M$ is a memory set  for $\tau$ and let $\mu_M \colon A^M \to A$ denote the associated local defining map.
Let $K$ be a field and let $X$ denote  the discrete $K$-scheme obtained by taking a direct union of a family of copies of the spectrum of $K$ indexed by $A$ (note that $X$ is not of finite type when the set $A$ is infinite).
Let $Y$ denote the spectrum of an algebraic closure of $K$, viewed as a $K$-scheme.
Then $A = X(Y)$.
Moreover, $\mu_M$ is induced by a unique $K$-scheme morphism $f \colon X^M \to X$.
 \par
In the rest of the paper, we shall mainly consider cellular automata over schemes $S$, $X$, $Y$, where the scheme $S$ is the spectrum of some field $K$,
the scheme $X$ is an algebraic variety over $K$ (that is, a $K$-scheme of finite type), and the scheme $Y$ is the spectrum of an algebraically closed  extension of $K$.
\par    
Recall that a  group $G$ is called \emph{reisidually finite} if the intersection of the finite-index subgroups of $G$ is reduced to the identity element. 
Every finitely generated abelian group (e.g.~$G = \Z^d$) and, more generally, by a theorem due to Mal'cev, every finitely generated linear group is residually finite. 
Here, by a \emph{linear} group, we mean a group that is isomorphic to a subgroup of 
$\GL_n(K)$ for some integer $n \geq 1$ and some field $K$.  
It was observed by Lawton 
(cf.~\cite{gottschalk}, \cite{lawton}) that, wen $G$ is a residually finite group and  $A$ is a finite alphabet, every injective cellular automata $\tau \colon A^G \to A^G$ is surjective.
  This property of residually finite groups, known as \emph{surjunctivity}, was extended to the class of \emph{sofic}  groups by Gromov~\cite{gromov-esav} and Weiss~\cite{weiss-sgds}.
  The class of sofic groups contains in particular all residually amenable groups and is therefore much larger than the class of residually finite groups.
Actually, there is no example of a non-sofic group up to now and it is not known if all groups are surjunctive.
In~\cite[Corollary~1.2]{cc-algebraic-ca}, it was shown that if $K$ is an uncountable algebraically closed  field,  
$V \subset K^n$ is an affine algebraic set over $K$, and $G$ is a residually finite group,
then every injective algebraic cellular automaton $\tau \colon V^G \to V^G$ is surjective.
Other surjunctivity results for cellular automata may be found in~\cite{gromov-esav}, \cite{cc-israel}, \cite{cc-artinian}.
Given a property $(P)$ of groups, one says that a group $G$ is \emph{locally}  $(P)$ if every finitely generated subgroup of $G$ satisfies $(P)$.
A group is sofic if and only if it is locally sofic.
On the other hand, there are groups, such as the additive group $\Q$ of rational numbers, that are locally residually finite but not residually finite.
By Mal'cev's theorem, every linear group is locally residually finite.
\par
In the setting of Definition \ref{def:ca-over-schemes}, it is natural to expect the surjunctivity result of Corollary~1.2 in~\cite{cc-algebraic-ca} to hold for other classes of algebraic varieties. 
In the present paper, we shall establish the surjunctivity property in the case of complete algebraic varieties. 
This class is ``orthogonal" to the class of affine algebraic varieties in the sense that every complete affine algebraic variety is finite.
Furthermore, 
if the base field is uncountable, we can remove the completeness assumption and thus cover also Corollary~1.2 in~\cite{cc-algebraic-ca}. 
More precisely, we get the following result. 

\begin{theorem}
\label{t:surj-ca}
Let $G$ be a locally residually finite  group and
let $X$ be an algebraic variety over an algebraically closed field $K$. 
Let $A \coloneqq X(K)$ denote the set of $K$-points of $X$.
  Suppose that the variety $X$ is complete or that the field $K$ is uncountable. 
  Then every injective  cellular automaton 
$\tau \colon A^G\to A^G$ 
over the group $G$ and the $K$-scheme $X$  with coefficients in $K$ 
is surjective (and hence bijective). 
\end{theorem}

Given a group $G$ and a  set $A$,  a cellular automaton $\tau \colon A^G \to A^G$ is called \emph{reversible} if $\tau$ is bijective and its inverse map $\tau^{-1} \colon A^G \to A^G$ is also a cellular automaton. 
When the alphabet set $A$ is finite, 
it is well known that every bijective cellular automaton 
$\tau \colon A^G \to A^G$ is reversible 
(see e.g.~\cite[Theorem~1.10.2]{ca-and-groups-springer} for a topological proof).
In~\cite[Theorem~1.3]{cc-algebraic-ca}, it was shown that
if $K$ is an uncountable algebraically closed field and $V$ an affine algebraic set  over $K$,
then  every bijective algebraic cellular automaton $\tau \colon V^G \to V^G$ is reversible.
Here we shall prove the following result which extends Theorem~1.3 in~\cite{cc-algebraic-ca}.

\begin{theorem}
\label{t:main-reversible}
Let $G$ be a group and
let   $X$ be an algebraic variety over an uncountable algebraically closed field $K$.
Let $A \coloneqq X(K)$ denote the set of $K$-points of  $X$.  
 Then any bijective cellular automaton $\tau \colon  A^G\to A^G$ over the group $G$ and the  
 $K$-scheme    $X$ with coefficients in   $K$  is reversible.  
\end{theorem}

When a cellular automaton over schemes is reversible, it is natural to ask whether the  inverse cellular automaton is also a cellular automaton over the same schemes.
We shall obtain the following partial positive answer to this question.

\begin{theorem}
\label{t:inverse-also}
Let $G$ be a locally residually finite group and 
let $X$ be a separated and reduced algebraic variety
over an algebraically closed field $K$ of characteristic $0$.
Let $A \coloneqq X(K)$ denote the set of $K$-points of  $X$.
Suppose that $\tau \colon A^G \to A^G$ is a
reversible cellular automaton over the group $G$ and the $K$-scheme $X$ with coefficients in $K$.
Then its inverse cellular automaton
$\tau^{-1} \colon A^G \to A^G$ is also a cellular automaton
over the $K$-scheme $X$ with coefficients in $K$.  
\end{theorem}

Concerning the invertiblity of affine algebraic cellular automata, 
it is asked in \cite[Question~(Q1)]{cc-algebraic-ca}
whether  
 there exists a bijective algebraic cellular automaton $\tau \colon A^\Z \to A^\Z$ over $\C$  whose inverse cellular automaton $\tau^{-1} \colon A^\Z \to A^\Z$ is not algebraic.
Combining Theorem~\ref{t:surj-ca}, Theorem~\ref{t:main-reversible}, and Theorem~\ref{t:inverse-also}, we obtain the following corollary which gives in particular a negative answer to the above question. 

\begin{corollary}
\label{c:answer-question}
Let $G$ be a locally residually finite group and let $X$ be a separated and reduced complex algebraic variety. 
Let $A \coloneqq X(\C)$. 
Suppose that $\tau \colon A^G \to A^G$ is an injective cellular automaton over the group $G$ and the complex scheme $X$ with complex coefficients. 
Then $\tau$ is reversible and its inverse is also a cellular automaton
over the complex scheme $X$ with complex coefficients.
\end{corollary}

The paper is organized as follows.
In Section~\ref{sec:ca} we fix notation and recall some facts about cellular automata.
General properties of cellular automata over schemes are investigated in Section~\ref{sec:ca-schemes}. It is shown in particular that the set of cellular automata over a group $G$ and schemes $S$, $X$,
and $Y$, form a monoid for the composition of maps. In Section~\ref{sec:cip-schemes},
we establish a fundamental property of the cellular automata 
over algebraic varieties that are studied here, namely that their image is closed for the prodiscrete topology on the set of configurations
(Theorem~\ref{t:closed-image}). 
Sections 7, 8,  and 9 contain the proofs of the three theorems stated above.
Counter-examples showing the necessity of certain of our hypotheses are discussed in the final section.
In Appendix \ref{sec:schemes} we collect, for the possible convenience of the reader, the necessary background material on schemes and algebraic varieties.
In Appendix \ref{sec:inverse}, we review inverse limits and establish an auxiliary result on the non-emptiness of inverse limits of constructible subsets of algebraic varieties.
\par
We are very indebted to Misha Gromov for all the great ideas contained in~\cite{gromov-esav}
which were a permanent source of inspiration and motivation during the preparation of the present paper.

\section{Cellular automata}
\label{sec:ca}

\subsection{The prodiscrete topology}
Let $A$ be a set and let $G$ be a group.
The \emph{prodiscrete topology} on  $A^G = \prod_{g \in G} A  $  is the product topology obtained by taking the discrete topology on each factor $A$ of $A^G$.
Observe that the prodiscrete topology on $A^G$ is Hausdorff since the product of a family of Hausdorff topological spaces is itself Hausdorff.
If $c \in A^G$, the sets
\[
V(c,\Omega) \coloneqq \{d \in A^G : c\vert_\Omega = d\vert_\Omega\},
\]
where $\Omega$ runs over all finite subsets of $G$,
form a base of open neighborhoods of $c$  for the prodiscrete topology on $A^G$.
\par
Suppose that  $\tau \colon A^G \to A^G$ is a cellular automaton with memory set $M$.
Formula~\eqref{e;local-property} implies that
if two configurations $c,d \in A^G$ coincide on $g M$ for some $g \in G$, then the  configurations 
$\tau(c)$ and $\tau(d)$ take the same value at $g$.
Therefore,
if $d \in V(c,\Omega M)$ for some finite subset $\Omega \subset G$,
then $\tau(d) \in V(\tau(c),\Omega)$.
This shows that
every cellular automaton $\tau \colon A^G \to A^G$ is continuous with respect to the prodiscrete topology on $A^G$.

\subsection{The closed image property for cellular automata}
\label{subsec:cip-ca}
Let $A$ be a set and let $G$ be a group. 
Let $\tau \colon A^G \to A^G$ 
be a cellular automaton over the group $G$ and the alphabet $A$, with memory set $M$ and associated local defining map 
$\mu_M \colon A^M \to A$.
\par
Consider  the directed poset $\PP$ consisting of all finite subsets of $G$, partially ordered by inclusion.
\par
Let  $d\in A^G$ be a configuration that is in the closure of the image of $\tau$ with respect to the prodiscrete topology on $A^G$.
\par
For any  $\Omega \in \PP$, we write $\Omega'\coloneqq \{g\in G \colon gM\subset \Omega\}$, 
so that 
\[
\Omega'= \bigcap_{h \in M} \Omega h^{-1} \in \PP.
\]
Note that if $\Omega,\Lambda \in \PP$  and $\Omega \subset \Lambda$, then 
$\Omega' \subset \Lambda'$.
\par
For each $\Omega \in \PP$, the cellular automaton induces a map
$\tau_\Omega \colon A^{\Omega} \to A^{\Omega'}$ defined by
\[
\tau_\Omega(u)(g) \coloneqq  (\tau(e))\vert_{\Omega'}
\]
for all $u \in A^\Omega$, where $e \in A^G$ is any configuration extending  $u$.
Observe that the map $\tau_\Omega$ is well defined since, for every $g \in \Omega'$, the value of $\tau(e)$ at $g \in \Omega'$ is entirely determined by the restriction of $e$ to
$g M \subset \Omega$.
\par 
Consider now the inverse system
$(Z_\Omega,\varphi_{\Omega,\Lambda})$ indexed by $\PP$, where
$Z_\Omega = Z_\Omega(d) \subset A^\Omega$ is defined by
\begin{equation}
\label{l:def-Z-Omega}
Z_\Omega \coloneqq \tau_\Omega^{-1}(d\vert_{\Omega'})
\end{equation}
and the transition maps $\varphi_{\Omega,\Lambda} \colon Z_\Lambda \to Z_\Omega$ 
are induced by the restriction maps
$\rho_{\Omega,\Lambda} \colon A^{\Lambda} \to A^{\Omega}$ for  all $\Omega,\Lambda \in \PP$ such that $\Omega \subset \Lambda$.
\par
Observe that
\begin{equation}
\label{e:Z-omega-not-empty}
Z_\Omega \not= \varnothing \text{  for all  } \Omega \in \PP.
\end{equation}
Indeed, by our hypothesis, $d$ is in the closure of the image of $\tau$ for the prodiscrete topology on $A^G$,  so that, for every $\Omega \in \PP$, there exists $c \in A^G$ such that 
$(\tau(c))\vert_{\Omega'} = d\vert_{\Omega'}$ and hence
$c\vert_{\Omega} \in Z_\Omega$.

\begin{lemma}
\label{l:charact-cip}
The following conditions are equivalent:
\begin{enumerate}[label=(\alph*)]
\item
$\tau$ has the closed image property with respect to the prodiscrete topology on $A^G$;
\item
for every configuration $d \in A^G$ that is in the closure of $\tau(A^G)$ with respect to the prodiscrete topology,
one has $\varprojlim Z_\Omega \not= \varnothing$. 
\end{enumerate}
\end{lemma}

\begin{proof}
Let $Z \coloneqq \varprojlim Z_\Omega$.
\par
If $d = \tau(c)$ for some $c \in A^G$, then
we clearly have that  $(c\vert_{\Omega})_{\Omega \in \PP} \in Z$.
This shows that (a) implies (b).
\par
Conversely, suppose that 
\[
 (z_\Omega)_{\Omega \in \PP}  \in \varprojlim Z_\Omega ,
\] 
then there exists a unique configuration
$c \in A^G$ such that $c\vert_\Omega = z_\Omega$ for all  $\Omega \in \PP$.
On the other hand, we have that $\tau(c) = d$ by~\eqref{l:def-Z-Omega}.
This shows that (b) implies (a). 
\end{proof}

\subsection{Restriction of a cellular automaton to a subgroup}
Let $A$ be a set and let $G$ be a group. 
Let $\tau \colon A^G \to A^G$ 
be a cellular automaton over the group $G$ and the alphabet $A$, with memory set $M$ and associated local defining map 
$\mu_M \colon A^M \to A$.
Suppose that $H$ is a subgroup of $G$ containing $M$.
Then the map $\sigma \colon A^H \to A^H$ defined by
\[
\sigma(c)(h) = \mu((h^{-1}c)\vert_M)
\quad \text{  for all  } c \in A^H \text{ and } h \in H,
\]
is a cellular automaton over the group $H$ and the alphabet $A$, with memory set $M$ and associated local defining map $\mu_M$.
One says that $\sigma$ is the cellular automaton obtained by \emph{restriction} of  
$\tau$ to $H$.
\par

\begin{lemma}
\label{l:induction}
Let $G$ be a group, $A$ a set, and $H$ a subgroup of $G$. 
Suppose that  $\tau \colon A^G \to A^G$ is a cellular automaton over $G$ admitting
a memory set  $M \subset H$
and let $\sigma \colon A^H \to A^H$ denote the cellular automaton over $H$ obtained by restriction. 
Then the following holds:
\begin{enumerate}[label=(\roman*)]
 \item 
$\tau$ is injective if and only if $\sigma$ is injective;
 \item 
$\tau$ is surjective if and only if $\sigma$ is surjective;
 \item 
$\tau$ is bijective if and only if $\sigma$ is bijective;
\item 
$\tau$ is reversible if and only if $\sigma$ is reversible;
 \item 
$\tau(A^G)$ is closed in $A^G$ for the prodiscrete topology  if and only if $\sigma(A^H)$ is closed in $A^H$ 
for the prodiscrete topology.
\end{enumerate} 
\end{lemma}

\begin{proof}
See \cite[Theorem~1.2]{induction}.
\end{proof}

\subsection{Periodic configurations}
Let $A$ be a set and let $G$ be a group.
\par
Let $H$ be a subgroup of $G$.
One says that a configuration $c \in A^G$ is \emph{fixed} by $H$ if $h c = c$ for all $h \in H$.
Denote by  $\Fix(H)$ the subset of $A^G$ consisting of all configurations  that are fixed by $H$. 
Consider the set $H \backslash G \coloneqq \{ Hg : g \in G \}$ consisting of all right cosets of $H$ in $G$ and the canonical surjective map
$\pi_H \colon G \to H \backslash G$
defined by $\pi_H(g) = Hg$ for all $g \in G$.
Given an element $d \in A^{H \backslash G}$, i.e., a map $d \colon H \backslash G \to A$,  the composite map
$d \circ \pi_H \colon G \to A$ is in $\Fix(H)$.
Moreover, the map $\rho_H \colon  A^{H \backslash G} \to \Fix(H)$, defined by 
\begin{equation}
\label{e:def-rho-H}
\rho_H(d) \coloneqq d \circ \pi_H \text{ for all }d \in A^{H \backslash G},
\end{equation}  
is bijective (see e.g.~\cite[Proposition~1.1.3]{ca-and-groups-springer}). 
\par
A configuration $c \in A^G$ is called \emph{periodic} if its orbit under the $G$-shift is finite.
This amounts to saying that there exists a finite-index subgroup  $H \subset G$ such that 
$c \in \Fix(H)$.

\begin{lemma}
\label{l:periodic-conf-are-dense}
Let $A$ be a set and let $G$ be a residually finite group.
Then the periodic configurations are dense in $A^G$ for the prodiscrete topology.
\end{lemma}

\begin{proof}
See e.g.~\cite[Theorem~2.7.1]{ca-and-groups-springer}.
\end{proof}

\section{Cellular automata over schemes}
\label{sec:ca-schemes}

\subsection{Change of memory set}
The following result shows that Definition~\ref{def:ca-over-schemes} does not depend of the choice of the memory set $M$ if $X(S) \not= \varnothing$. 

\begin{proposition}
\label{p:independent-of-memory}
Let $G$ be a group.  
Let $S$ be a scheme and let $X, Y$ be schemes over  $S$
with $X(S) \not= \varnothing$.
Let $A \coloneqq X(Y)$ denote the set of $Y$-points of $X$
and  suppose that $\tau \colon A^G \to A^G$ is a cellular automaton.
Then the following conditions are equivalent:
\begin{enumerate}[label=(\alph*)]
\item
there exists a memory set $M$ of $\tau$ such that 
the associated local defining map  $\mu_M \colon A^M \to M$ 
satisfies $\mu_M = f^{(Y)}$  for some $S$-scheme morphism $f \colon X^M \to X$; 
\item
for any memory set $M$ of $\tau$,
the associated local defining map  $\mu_M \colon A^M \to A$ 
satisfies $\mu_M = f^{(Y)}$  for some $S$-scheme morphism $f \colon X^M \to X$. 
\end{enumerate}
\end{proposition}

\begin{proof}
It suffices to prove that (a) implies (b) since 
the converse implication is trivial.
Let $M$ be a memory set of $\tau$ and
suppose that the local defining map $\mu_M \colon A^M \to A$ 
satisfies $\mu_M = f^{(Y)}$  for some $S$-scheme morphism $f \colon X^M \to X$. 
Let $N$ be another memory set of $\tau$ and let us show that the associated local defining map 
$\mu_{N} \colon A^{N} \to A$ 
satisfies $\mu_N = g^{(Y)}$  for some $S$-scheme morphism $g \colon X^M \to X$. 
 Let $M_0$ denote the minimal memory set of $\tau$
 (see e.g.~\cite[Section~1.5]{ca-and-groups-springer}).
 Recall that  $M_0 \subset M \cap N$.
Consider the projection  map 
 $\rho \colon A^{N} \to A^{M_0}$.
Observe that
\begin{equation}
\label{e:rho-ind-r} 
\rho =  r^{(Y)},
\end{equation} 
where  $r \colon X^{N} \to X^{M_0}$ is the projection $S$-scheme morphism.
Moreover, we  have that
\begin{equation}
\label{e;projectioon-minimal-m-s}
 \mu_{N} = \mu_{M_0} \circ \rho.  
\end{equation}
Since $X(S)  \not= \varnothing$ by our hypothesis,
there exists  an $S$-scheme morphism $e \colon S \to X$.
Consider the $S$-scheme morphism 
\[
h \colon X^{M_0} = X^{M_0} \times S^{M \setminus M_0} \to X^M
\]
that is the $S$-fibered product of the identity morphism $X^{M_0} \to X^{M_0}$
with a family, indexed by $M \setminus M_0$, of copies of the morphism $e \colon S \to X$.
Clearly $\mu_{M_0} = (f \circ h)^{(Y)}$. 
By using~\eqref{e;projectioon-minimal-m-s} and~\eqref{e:rho-ind-r}, we get  
\[
\mu_{N} = (f \circ h)^{(Y)} \circ r^{(Y)} = (  f \circ h \circ r)^{(Y)},
\]
so that we can take $g = f \circ h \circ r$.
This shows that (a) implies (b). 
\end{proof}

\subsection{The monoid $\CA(G,S,X,Y)$}

\begin{lemma}
\label{l:monoid}
Let $G$ be a group and let $M, M'\subset G$ be finite subsets.  
Let $S$ be a scheme and let $X, Y$ be schemes over $S$.
Let $A \coloneqq X(Y)$ denote the set of $Y$-points of $X$. 
Suppose that $\tau, \tau' \colon A^G \to A^G$ are  cellular automata
over the group $G$ and the alphabet $A$
with respective memory sets $M,M'$  
whose local defining maps $\mu , \mu' $ are induced respectively by $S$-scheme morphisms $f \colon X^M \to X$ and $f' \colon X^{M'} \to X$. 
Then the composite map  $\tau' \circ \tau \colon A^G \to A^G$ is a cellular automaton with memory set $M' M$ whose local defining map is induced by a morphism of $S$-schemes $X^{M' M} \to X$.
\end{lemma}

\begin{proof}
By \cite[Proposition~1.4.9]{ca-and-groups-springer}, the composite map  
$\tau' \circ \tau$ is a cellular automaton with memory set $M' M$ whose  associated local defining map 
$\eta $  can be obtained as a composition $A^{M'M} \to A^{M} \to A$ 
in the following way.  
For each $y \in A^{M'M}$ and $g' \in M'$, define  $y_{g'} \in A^M$ 
 by $y_{g'}(g)=y(g'g)$ for all  $g\in M$. 
Let $\widetilde{y} \in A^{M'}$ be defined by 
$\widetilde{y}(g')=\mu(y_{g'})$ for all $g' \in M'$. 
This defines a map $\xi \colon A^{M'M} \to A^{M'}$, where
 $\xi(y) \coloneqq  \widetilde{y}$ for all $yy \in A^{M' M}$. 
Then, we have that  $\eta = \mu' \circ \xi$.
\par
We will construct in parallel a morphism of $S$-schemes $X^{M' M} \to X$ which induces $\eta$. 
For each $g' \in M'$, let $p_{g'} \colon M \to M'M$ be the map given by $g \mapsto g'g$ for all $g\in M$. 
By the universal property of fibered products, let $h \colon X^{M'M} \to X^{M'}$ be the $S$-scheme morphism induced by the family of $S$-scheme morphisms $f \circ p^*_{g'}\colon X^{M'M} \to X$,  
where $g'$ runs over $M'$ (cf.~Remark~\ref{rem:X-E-X-F}). 
It is clear that $h$ induces the  map $\xi \colon A^{M' M} \to A^{M'}$ defined above.  
Therefore, the $S$-scheme morphism $f' \circ h \colon  X^{M' M}\to X$ induces the local defining map 
$\eta \colon A^{M' M} \to A$ of $\tau' \circ \tau$.
\end{proof}

Let $G$, $S$, $X$, $Y$, and $A$ as in the above lemma.
We denote by $\CA(G,S,X,Y)$ the set consisting of all cellular automata $\tau \colon A^G \to A^G$ over the group $G$ and the schemes $S$, $X$, $Y$.
Observe that the identity map $A^G \to A^G$ is in $\CA(G,S,X,Y)$ since it is a cellular automaton with memory set $\{1_G\}$ whose associated local defining map is induced by the identity $S$-scheme morphism $X \to X$.
Consequently, from Lemma~\ref{l:monoid} we immediately deduce  the following. 

\begin{proposition}
The set $\CA(G,S,X,Y)$ is a monoid for the composition of maps.
\end{proposition}

We say that a cellular automaton $\tau \in \CA(G,S,X,Y)$ is \emph{invertible}
if $\tau$ is an invertible element of the monoid $\CA(G,S,X,Y)$, i.e., there exists 
$\sigma \in \CA(G,S,X,Y)$
such that $\sigma \circ \tau = \tau \circ \sigma$ is the identity map on $X(Y)^G$.

\section{The closed image property}
\label{sec:cip-schemes}

The goal of this section is to establish the following closed image theorem.

\begin{theorem}
\label{t:closed-image}
Let $G$ be a group and let $K$ be an algebraically closed  field. 
Let $X$ be an algebraic variety over $K$ and 
let $A \coloneqq X(K)$ denote the set of $K$-points of $X$.
 Suppose that 
the variety $X$ is complete 
or that the field $K$ is uncountable.
Let $\tau \colon  A^G\to A^G$ be a cellular automaton over the group $G$ and the $K$-scheme $X$ with coefficients in $K$.
Then $\tau$ has the closed image property, that is, $\tau(A^G)$ is closed in $A^G$ for the prodiscrete topology. 
\end{theorem}

For the proof, we consider a configuration $d \in A^G$ that is in the closure of $\tau(A^G)$ with respect to the prodiscrete topology on $A^G$, and keep the notation introduced in 
Subsection~\ref{subsec:cip-ca}.
We want to show that $d$ is in the image of $\tau$.
\par
Since $K$ is algebraically closed,
the set $X^\Omega(K)$ can be identified, for every $\Omega \in \PP$,
with the set of closed points of the $K$-algebraic variety $X^\Omega$
(see Remark~\ref{rem:closed-points-rat-points}).
Thus, we have  natural inclusions
\[ 
Z_\Omega \subset A^{\Omega} = (X(K))^{\Omega} = X^{\Omega}(K) \subset X^{\Omega},
\]  
so that we can  equip $A^\Omega$  and $Z_\Omega$ with the topology induced by the 
$K$-algebraic variety $X^{\Omega}$.
\par
For every $\Omega \in \PP$, the space $Z_\Omega$ is non-empty  
by~\eqref{e:Z-omega-not-empty} and $T_0$ since the underlying topological space of any scheme is $T_0$.
Note also that $Z_\Omega$ is quasi-compact by Lemma~\ref{l:prop-alg-var}.(ii) since $Z_\Omega \subset X^\Omega$ and $X^\Omega$ is an algebraic variety.
\par
For all $\Omega, \Lambda \in \PP$ with $\Omega \subset \Lambda$,
the transition map $\varphi_{\Omega,\Lambda} \colon Z_\Lambda \to Z_\Omega$  is  continuous since it is   obtained by restricting the  $K$-scheme projection morphism
$X^\Lambda \to X^\Omega$.
\par
We claim that $\tau_\Omega \colon A^\Omega \to A^{\Omega'}$ is induced by a $K$-scheme morphism
$X^{\Omega} \to X^{\Omega'}$.  
To see this, consider, for each $g \in \Omega'$,
the $K$-scheme projection morphism  $p_g \colon X^\Omega \to X^{g M}$  
and the $K$-scheme isomorphism  $i_g \colon X^{gM} \to X^M$  induced by the bijective map $g M \to  M$ given by left multiplication by $g^{-1}$.
Then the family of  $K$-scheme morphisms $f \circ i_g \circ p_g \colon X^{\Omega} \to X$, $g \in \Omega'$,
 yields, by the universal property of $K$-fibered products,
 a  $K$-scheme morphism $f_\Omega \colon X^{\Omega} \to X^{\Omega'}$.
It is clear from this construction that $f_\Omega$ extends $\tau_\Omega$.
This proves our claim and shows in particular that $\tau_\Omega$ is continuous.
\par
We shall first establish  Theorem~\ref{t:closed-image} in the case when $X$ is complete.

\begin{proof}[Proof of Theorem~\ref{t:closed-image} when $X$ is complete]
Suppose that $X$ is complete, i.e., the $K$-scheme $X$ is proper.
\par
Let us check that all the hypotheses of  Stone's theorem (cf.~Theorem~\ref{t:stone})  
are satisfied by  the inverse system $(Z_\Omega)_{\Omega \in \PP}$.
\par
We have already seen that the topological spaces $Z_\Omega$ are non-empty quasi-compact $T_0$ spaces and that the transition maps
$\varphi_{\Omega,\Lambda} \colon Z_\Lambda \to Z_\Omega$ are continuous.
\par 
It only remains to verify that the transition maps $\tau_{\Omega,\Lambda}$ are closed
for all $\Omega,\Lambda \in \PP$ with $\Omega \subset \Lambda$.
\par
Since $X$ is proper and the fibered product of proper schemes is proper, 
the $K$-scheme $X^{\Omega}$ is proper and hence quasi-compact and Jacobson
for every $\Omega \in \PP$.
As $A^{\Omega'} = X^{\Omega'}(\overline{K})$ is the set of closed points of $X^{\Omega'}$ and
$d\vert_{\Omega'} \in A^{\Omega'}$,
the continuity of $\tau_\Omega$ implies that
$Z_\Omega$ is closed in $A^{\Omega'}$.
\par
On the other hand, for all $\Omega,\Lambda \in \PP$ with $\Omega \subset \Lambda$,
the transition map  $\varphi_{\Omega,\Lambda} \colon Z_\Lambda \to Z_\Omega$ is the restriction of the $K$-scheme projection
$X^\Lambda \to X^\Omega$.
As morphisms between proper schemes are proper and therefore closed, 
we deduce that
all transition maps $\varphi_{\Omega,\Lambda}$ are closed. 
\par
It therefore follows from  Stone's theorem
 that $\varprojlim Z_\Omega \not= \varnothing$.
 We conclude that $\tau$ has the closed image property with respect to the prodiscrete topology on $A^G$ by applying 
 Lemma~\ref{l:charact-cip}.
 \end{proof}

To complete the proof of Theorem~\ref{t:closed-image},
it remains to treat the case when the field $K$ is uncountable.

\begin{proof}[Proof of Theorem~\ref{t:closed-image} when $K$ is uncountable]
Let us assume  that the field $K$ is uncountable.
\par
Let $H$ denote the subgroup of $G$ generated by the memory set $M$.  Consider the cellular automaton 
$\sigma \colon A^H \to A^H$ over the alphabet $A$ and the group $H$ obtained by restricting $\tau$ to $H$. Observe that $H$ is finitely generated and therefore countable.
On the other hand,
$\sigma$ is also a cellular automaton over the $K$-scheme $X$
with coefficients in $\overline{K}$.
Indeed, it admits the same memory set $M$ and the associated   local defining map is also induced by 
$f\colon X^M\to X$.
As it follows from Assertion~(v) in Lemma~\ref{l:induction}  that $\tau(A^G)$ is closed in $A^G$  for the prodiscrete topology on $A^G$  if and only if $\sigma(A^H)$ is closed in $A^H$ for the prodiscrete topology on $A^H$,
we can assume in the sequel, without loss of generality,
 that the group $G$ is countable. 
\par 
Since $G$ is countable, there is a sequence  $(\Omega_n)_{n \in \N}$ of finite subsets of $G$ such that $G = \bigcup_{n \in \N} \Omega_n$ and $\Omega_n \subset \Omega_{n + 1}$ for all $n \in \N$.
Consider the inverse sequence $(Z_n,\varphi_{m n})$, where
\[
Z_n \coloneqq Z_{\Omega_n}
\]
and
\[
\varphi_{n m} \coloneqq \varphi_{\Omega_n,\Omega_m} \colon Z_m \to Z_n 
\]
for all $n,m \in \N$ such that $n \leq m$.
Note that the inverse sequence $(Z_n,\varphi_{n m})$ has the same inverse limit as the inverse system
$(Z_\Omega,\varphi_{\Omega,\Lambda})$ since the set consisting of all $\Omega_n$,  $n \in \N$, is cofinal in $\PP$.
\par
We observe now that, for each $n\in \N$, the set 
\[
C_n\coloneqq f_{\Omega_n}^{-1}(d\vert_{\Omega_n'})\subset X^{\Omega_n}
\] 
is a closed and hence constructible subset of the $K$-algebraic variety $X^{\Omega_n}$
and that  $Z_n$ is the set of closed points of $C_n$. 
Thus, by using Lemma~\ref{l:inverse-limit-seq-const}, we conclude that 
$\varprojlim Z_\Omega = \varprojlim Z_n \neq \varnothing$.
 By Lemma~\ref{l:charact-cip},
this shows that $d$ is in the image of $\tau$.
Thus $\tau$ has the closed image property for the prodiscrete topology on $A^G$.
\end{proof}

\section{Surjunctivity}
\label{s:construction}

The aim of this section is to prove the following result 
which extends Theorem~\ref{t:surj-ca}.

\begin{theorem}
\label{t:surj-ca-general}
Let $G$ be a locally residually finite group with a finite subset $M$. 
Let $K$ be a field with an algebraic closure $\overline{K}$. 
Let $X$ be a $K$-algebraic variety 
and let $f\colon X^M \to X$ be a morphism of $K$-schemes.
For each $K$-scheme $Y$, 
let $\tau^{(Y)}$ denote the cellular automaton
over the group $G$ and the alphabet $X(Y)$ 
with memory set $M$ and associated local defining map $f^{(Y)}$. 
Suppose that $X$ is complete or that $K$ is uncountable. 
Suppose also that  $\tau^{(L)}$ is injective for some field extension $L/\overline{K}$. 
Then the following hold:
\begin{enumerate}[label=(\roman*)]
\item
$\tau^{(L')}$ is surjective for every algebraically closed field extension $L'/K$. 
\item
$\tau^{(K')}$ is injective for every field extension $L /K' /K$.  
\end{enumerate}
\end{theorem}

\begin{proof}[Proof of Theorem~\ref{t:surj-ca} from Theorem~\ref{t:surj-ca-general}] 
It suffices to take $L=L' = \overline{K} = K$ in Theorem \ref{t:surj-ca-general}.(i) to conclude. 
\end{proof}
Before giving the proof of Theorem~\ref{t:surj-ca-general},
we first introduce the following general construction. 
Consider a  cellular automaton $\tau\colon A^G \to A^G$ defined over the schemes $S,X,Y$. 
Thus, $X$ and $Y$ are $S$-schemes,  
$A=X(Y)$, and $\tau$ admits a memory set $M$ such that the associated local defining map 
$\mu \colon A^M \to A$ satisfies $\mu = f^{(Y)}$ 
for some $S$-scheme  morphism  $f\colon X^M \to X$.
Let $H$ be a finite index subgroup of $G$ and let $\Fix(H)\subset A^G$ denote the set of configurations that are fixed by $H$.
Consider the bijective map $\rho_H \colon A^{H \backslash G} \to \Fix(H)$ defined 
by~\eqref{e:def-rho-H}. 
For each right coset $\gamma \in H \backslash G$, we consider
the map
$p_{\gamma} \colon M \to H \backslash G$, defined by 
$p_{\gamma}(h) \coloneqq \gamma h$ for all $h \in M$,
and 
 the induced morphism of $S$-schemes
\[
\varphi_{\gamma} \coloneqq p_{\gamma}^* \colon X^{H \backslash G} \to X^M
\]
(cf.~Remark~\ref{rem:X-E-X-F}). 
Observe that $\varphi_{\gamma}^{(Y)} (\overline{z})=(g^{-1}z)\vert_M$ for all $g\in \gamma$ and 
$\overline{z} \in A^{H\backslash G}$ such that $z=\rho_H(\overline{z})$.
By the universal property of fibered products, 
the family of $S$-scheme morphisms $f\circ \varphi_{\gamma}  \colon X^{H \backslash G} \to X$, where  $\gamma$ runs over  $H\backslash G$,
induces an $S$-scheme morphism
$\widetilde{\tau_H} \colon X^{H \backslash G} \to X^{H \backslash G}$.
Let $\overline{z}\in A^{H \backslash G}$ and $z=\rho_H(\overline{z})$.
 Then, for all $\gamma \in H \backslash G$ and $g \in \gamma$,
  we have that
\[
\widetilde{\tau_H}^{(Y)}(\overline{z})(\gamma)=\mu \circ \varphi_{\gamma}^{(Y)} (\overline{z})=\mu((g^{-1}z)\vert_M),
\]
and
\[
\rho_H^{-1}\circ \tau\circ \rho_H(\overline{z})(\gamma)=\rho_H^{-1}\circ \tau(z)(\gamma)=\tau(z)(g)=\mu ((g^{-1}z)\vert_M).
\]
Therefore,  
\begin{equation}
\label{e:conjugate}
\widetilde{\tau_H}^{(Y)} =\rho_H^{-1}\circ \tau\vert_{\text{Fix}(H)} \circ \rho_H.
\end{equation}
In other words, the diagram
\[
\begin{CD}
A^{H \backslash G} @>{\widetilde{\tau_H}^{(Y)}}>> A^{H \backslash G}   \\
@V{\rho_H}VV @VV{\rho_H}V \\
\text{Fix}(H) @>{\tau\vert_{\text{Fix}(H)}}>> \text{Fix}(H)
\end{CD}
\]
is commutative.

\begin{proof}[Proof of Theorem~\ref{t:surj-ca-general}]
Let $Y$ be a $K$-scheme.
Let $G'$ denote the subgroup of $G$ generated by $M$ and let  
$\sigma^{(Y)} \colon X(Y) ^{G'} \to X(Y)^{G'}$ be  the restriction of 
$\tau^{(Y)}$ to $G'$.
Then  $\sigma^{(Y)}$ is a cellular automaton over the group $G'$ with memory set
$M$ and associated local defining map $f^{(Y)}$.
Note that $G'$ is residually finite since $G$ is locally residually finite.
On the other hand, it follows from Assertions (i) and (ii) in Lemma~\ref{l:induction} that  $\tau^{(Y)}$ is injective (resp.surjective) if and only if 
$\sigma^{(Y)}$ is injective (resp.~surjective). 
Therefore, up to replacing $G$ by $G'$, we can assume that $G$ is residually finite.
\par
Let $\FF$ denote the set of all finite index subgroups of $G$. 
For $H\in \FF$, consider the morphism of $K$-schemes 
$\widetilde{\tau_H} \colon X^{H \backslash G} \to X^{H \backslash G}$ 
defined above with $S=\Spec(K)$.
\par 
For (ii), consider any field extension $L /K' /K$. 
We then have a natural inclusion $X(K') \subset X(L)$ and thus $\tau^{(K')}$ is the restriction of $\tau^{(L)}$ to $X(K')^G$. 
Since $\tau^{(L)}$ is injective, $\tau^{(K')}$ is injective as well and this proves (ii). 
In particular, by taking $K'=\overline{K}$, we see that $\tau^{(\overline{K})}$ is injective. 
As $\widetilde{\tau_H}^{(\overline{K})}$ is conjugate to $\tau^{(\overline{K})}\vert_{\Fix(H)}$
by~\eqref{e:conjugate},
we deduce that $\widetilde{\tau_H}^{(\overline{K})}$ 
is also injective. 
Now, since $X(\overline{K})^{H \backslash G}$ is the set of closed points of $X^{H \backslash G}$ by Lemma~\ref{l:closed-points-alg-closed-var}, 
it follows from  Lemma~\ref{l:mor-var-closed-points}.(i) that 
$\widetilde{\tau_H} \colon X^{H \backslash G} \to X^{H \backslash G}$ 
is an injective $K$-scheme endomorphism of the $K$-algebraic variety $X^{H \backslash G}$.
Therefore, $\widetilde{\tau_H}$ is surjective by the Ax-Grothendieck theorem 
(cf.~Theorem~\ref{t:surj-grothendieck}). 
\par
For (i), consider an algebraically closed field extension $L'/K$. 
Note that the alphabet of $\tau^{(L')}$ is $A=X(L')=X_{L'}(L')$. 
Observe that $X_{L'}$ is a complete $L'$-algebraic variety if $X$ is complete and 
that $L'$ is uncountable if $K$ is uncountable.  
Since  surjectivity  is preserved under base change (cf.~Lemma~\ref{l:surj-base-change}), $\widetilde{\tau_H}\times \Id_{L'}$ is also a surjective morphism of $L'$-algebraic varieties. 
By applying Lemma~\ref{l:mor-var-closed-points}.(ii), we deduce that 
$\widetilde{\tau_H}^{(L')}= (\widetilde{\tau_H} \times \Id_{L'})^{(L')}$ is surjective as well.
This implies that $\tau^{(L')}\vert_{\Fix(H)}$ is also surjective by~\eqref{e:conjugate}.
Thus $\Fix(H)= \tau^{(L')}(\Fix(H))\subset \tau^{(L')}(A^G)$.
Let   $E \subset A^G$ denote the set of periodic configurations. 
Then we have that $E=\bigcup_{H\in \FF}\Fix(H) \subset \tau^{(L')}(A^G)$.
On the other hand, since the group $G$ is residually finite, 
the set $E$ is dense in $A^G$ for the prodiscrete topology 
by virtue of Lemma~\ref{l:periodic-conf-are-dense}. 
As $\tau^{(L')}(A^G)$ is closed in $A^G$ for the prodiscrete topology by Theorem~\ref{t:closed-image}, 
we conclude that $\tau^{(L')}(A^G)=A^G$. Thus $\tau^{(L')}$ is surjective and this completes the proof of (i). 
\end{proof}
\section{Reversibility}

\begin{proof}[Proof of Theorem~\ref{t:main-reversible}]
Let $M$ be a memory set of $\tau$ such that the associated local defining map $\mu \colon A^M \to A$ is induced by a morphism of $K$-schemes $f \colon X^M \to X$.  
We can suppose that $1_G \in M$. 
To prove that $\tau$ is reversible, it suffices to show that the following condition is satisfied:
\begin{align*}
\tag{$*$}  
& \textit{There exists a finite set } N\subset G \textit{ such that for any }y\in A^G, \\
& \tau^{-1}(y)(1_G) \textit{ only depends on } y \vert_N \in A^N.
\end{align*}
Indeed, if $(*)$ holds, then there is a map $\nu\colon A^N\to A$ such that 
$\tau^{-1}(d)(1_G) =\nu(d\vert_N)$ for all $d\in A^G$. 
As $\tau$ is $G$-equivariant,
i.e., $\tau(g c) = g \tau(c)$ for all $g \in G$ and $c \in A^G$
(this immediately follows from Formula~\eqref{e;local-property}),
its inverse  
$\tau^{-1}$ is also $G$-equivariant. 
Hence, for all $d \in A^G$, we have that
\[
\tau^{-1}(d)(g)=g^{-1}\tau^{-1}(d)(1_G)=\tau^{-1}(g^{-1}d)(1_G)=\nu(g^{-1}d\vert_N),
\]
which implies that $\tau^{-1}$ is  reversible with memory set $N$ and associated local defining 
map  $\nu$.
\par
To prove $(*)$, we proceed by contradiction.
So let us suppose  that  $(*)$ is not satisfied.
 Then, using the notation introduced at the beginning of Subsection~\ref{subsec:cip-ca},
 for each $\Omega \in \PP$, there exist  $e_{\Omega},e'_{\Omega}\in A^G$ such that 
\begin{equation} \label{condition}
e_{\Omega}\vert_{\Omega'}=e'_{\Omega}\vert_{\Omega'} \quad \text{and} \quad \tau^{-1}(e_{\Omega})(1_G)\neq \tau^{-1}(e'_{\Omega})(1_G). 
\end{equation}
Let  $\PP_M$ be the subset of $\PP$ consisting of all finite subsets of $G$ 
that contain $M$. 
Let $\Omega \in \PP_M$. 
Consider the finite subset $\Omega' \subset G$,
the map $\tau_\Omega \colon A^{\Omega}\to A^{\Omega'}$,
and the
morphism of $K$-schemes $f_{\Omega} \colon X^{\Omega} \to X^{\Omega'}$
defined in Subsection \ref{subsec:cip-ca}. 
Recall in particular that  
$f_{\Omega}^{(K)}=\tau_{\Omega}$   
and  that $\tau_{\Omega}(c\vert_{\Omega})=\tau(c)\vert_{\Omega'}$ for all $c\in A^G$.  
Consider now the morphisms of $K$-schemes
\[
F_{\Omega}\coloneqq f_{\Omega} \times f_{\Omega} \colon  X^{\Omega}\times_K  X^{\Omega}\to  X^{\Omega'}\times_K  X^{\Omega'}
\] 
and the projection morphisms 
$\pi_{\Omega} \colon X^{\Omega}\times_K  X^{\Omega}\to X^{1_G} \times_K  X^{1_G}$.  
Denote by $\Delta\colon X^{1_G} \to X^{1_G}\times_K X^{1_G}$ and $\Delta_{\Omega'}\colon X^{\Omega'} \to  X^{\Omega'}\times_K  X^{\Omega'}$ the diagonal morphisms. 
We define a subset $V_{\Omega} \subset X^{\Omega}\times_K  X^{\Omega}$ by
\[
V_{\Omega}\coloneqq F_{\Omega}^{-1}(\Delta_{\Omega'}(X^{\Omega'}))\backslash \pi_{\Omega}^{-1}(\Delta(X^{1_G})).
\]
By Chevalley's theorem, $\Delta_{\Omega'}(X^{\Omega'})$ and $\Delta(X^{1_G})$ are constructible sets. 
Since the inverse image by a continuous map of a constructible set is constructible and the set difference of two constructible sets is constructible, $V_{\Omega}$ is also constructible. 
Observe that the closed points of $F_{\Omega}^{-1}(\Delta_{\Omega'}(X^{\Omega'}))$ are 
the pairs $(u,v)\in  A^{\Omega}\times A^{\Omega}$ such that $\tau_{\Omega}(u)=\tau_{\Omega}(v)$. 
Similarly, the set of closed points of 
$\pi_{\Omega}^{-1}(\Delta(X^{1_G}))$ are the pairs  
$(u,v)\in  A^{\Omega}\times A^{\Omega}$ such that $u(1_G)= v(1_G)$. 
Therefore, we deduce by \eqref{condition} that the set of closed points $Z_{\Omega}$ of $V_{\Omega}$ is nonempty for all $\Omega \in \PP_M$.
\par 
For all $\Omega, \Lambda \in \PP_M$ with  $\Omega \subset \Lambda$,  
consider the projection morphism 
$p_{\Omega \Lambda } \colon X^{\Lambda}\to X^{\Omega}$.
Then the morphism of $K$-schemes
\[
\pi_{\Omega \Lambda} \coloneqq p_{\Omega \Lambda}\times p_{\Omega \Lambda} \colon X^{\Lambda}\times_K  X^{\Lambda} \to X^{\Omega}\times_K  X^{\Omega},
\] 
induces by restriction a map $h_{\Omega \Lambda}\colon V_{\Lambda} \to V_{\Omega}$ and hence a map $\varphi_{\Omega \Lambda}\colon Z_{\Lambda} \to Z_{\Omega}$. 
Therefore, we obtain an inverse system 
$(Z_{\Omega},\varphi_{\Omega \Lambda})$ indexed by $\PP_M$.
\par
Suppose that $G$ is countable. 
Hence, there is a sequence $(\Omega_n)_{n\in \N}$ of finite subsets of $G$, each containing $M$, such that $G=\bigcup_{n\in N}\Omega_n$ and $\Omega_n\subset \Omega_{n+1}$ for all $n\in \N$. 
For $n, m\in \N$ such that $m\geq n$, define 
\[
V_n\coloneqq V_{\Omega_n},\quad Z_n\coloneqq Z_{\Omega_n},\quad h_{nm}\coloneqq h_{\Omega_n \Omega_m}.
\] 
We then have an \textit{inverse sequence} $(V_n,h_{nm})$ of nonempty constructible sets of $K$-algebraic varieties. 
Note that $Z_n$ are the sets of closed points of $V_n$ and that $h_{nm}$ are induced by morphisms of $K$-algebraic varieties. 
Since $K$ is uncountable, Lemma \ref{l:inverse-limit-seq-const} implies that $\varprojlim Z_n \neq \varnothing$. 
Since $(\Omega_n)_{n\in \N}$ is confinal in $\PP$, we have 
\[
\varprojlim Z_{\Omega} =\varprojlim Z_n \neq \varnothing.
\]
Let us choose $(w_{n})_{n \in \N}\in \varprojlim Z_{n}$. 
We can write $w_{n}=(u_{n},v_{n})\in (A^{\Omega_n})^2$.
Observe that $(u_{n})_{n\in \N}, (v_{n})_{n\in \N} \in \varprojlim A^{\Omega_n}$. 
Since $G=\bigcup_{n\in \N} \Omega_n$, there exist $c,c'\in A^G$ such that $c\vert_{n}=u_{n}$ and $c'\vert_{n}=v_{n}$ for all $n \in \N$.
By construction of the inverse system $Z_{n}$, we have that $c(1_G)\neq c'(1_G)$. 
Therefore, $c\neq c'$.  
Again by construction, we see that, for all $n \in \N$,
\[
\tau(c)\vert_{\Omega_n'}=\tau_{\Omega_n}(c\vert_{\Omega_n})=\tau_{\Omega_n}(c'\vert_{\Omega_n})=\tau(c')\vert_{\Omega_n'}.
\]
As $G=\bigcup_{n\in \N}\Omega_n'$, we deduce that $\tau(c)=\tau(c')$, which contradicts the injectivity of $\tau$. This completes the proof when $G$ is countable. 
\par 
For an arbitrary group $G$, let $H$ denote the countable subgroup of $G$ generated by $M$. 
The restriction $\sigma \colon A^H\to A^H$ of $\tau$ to $H$ is a cellular automaton over the $K$-scheme $X$ with coefficients in $\overline{K}$. 
Indeed, it has the same memory set $M$ and its local defining map is also induced by $f\colon X^M\to X$. 
By Lemma \ref{l:induction}.(iii), $\tau_H$ is also bijective. 
The above paragraph shows that $\sigma$ is reversible. 
Therefore, we can conclude by applying Lemma~\ref{l:induction}.(iv) that $\tau$ is reversible.
\end{proof} 

\section{Invertibility}

In this section, we shall establish the following result which extends
Theorem~\ref{t:inverse-also}. 

\begin{theorem}
\label{t:inverse-also-general}
Let $G$ be a locally residually finite group with a finite subset $M$. 
Let $K$ be a field of characteristic $0$ with algebraic closure $\overline{K}$.  
Suppose that $X$ is a separated and reduced $K$-algebraic variety. 
Let $f \colon X^M \to X$ be a $K$-scheme morphism.
For each $K$-scheme $Y$, 
let $\tau^{(Y)}$ denote the cellular automaton
over the group $G$ and the alphabet $X(Y)$ 
with memory set $M$ and associated local defining map $f^{(Y)}$. 
Then the following conditions are equivalent:
\begin{enumerate}[label=(\alph*)]
\item
  $\tau^{(L)}$ is reversible for some field extension $L/\overline{K}$;
  \item 
 $\tau^{(Y)}$ is reversible for every $K$-scheme $Y$. 
 \end{enumerate} 
Moreover, if these two equivalent conditions are satisfied, 
then there is a finite subset $N \subset G$ and a $K$-scheme morphism  $h \colon X^{N}\to X$ such that, for every $K$-scheme $Y$,
 the inverse cellular automaton  $(\tau^{(Y)})^{-1}$ is the cellular automaton
 with memory set $N$ whose associated local defining map is induced by $h$.
In particular, the cellular automaton
$\tau^{(Y)}$ is invertible in the monoid $\CA(G,K,X,Y)$.  
\end{theorem}

\begin{proof} [Proof of Theorem~\ref{t:inverse-also} from Theorem \ref{t:inverse-also-general}] 
It suffices to take $L=K=\overline{K}$ in Theorem~\ref{t:inverse-also-general} to conclude.
\end{proof}

For the proof of Theorem~\ref{t:inverse-also-general}, we shall need some auxiliary results. 

\begin{lemma}
\label{l:equalizer}
Let $K$ be a field with algebraic closure $\overline{K}$. 
Let $L/\overline{K}$ be a field extension. 
Let $f, g \colon X \to Y$ be $K$-scheme morphisms between  $K$-algebraic varieties. 
Suppose that $Y$ is separated over $K$ and that $X$ is reduced.
If the induced set maps $f^{(L)},g^{(L)} \colon X(L) \to Y(L)$ are equal, then $f$ and $g$ are equal as morphisms of $K$-schemes.
\end{lemma}

\begin{proof}
Since $Y$ is a separated scheme over $K$, the diagonal morphism $\Delta=(\Id_Y,\Id_Y) \colon Y \to Y \times_K Y$ is a closed immersion. 
Let $p_1 \colon Y\times_K Y \to Y$ and $p_2\colon Y\times_K Y \to Y$ denote  respectively the first and 
the second projection morphism. 
Consider the morphism $h\coloneqq (f,g)\colon X \to Y\times_K Y$.  
Then  $p_1\circ h=f \colon X \to Y$ and $p_2 \circ h=g \colon X \to Y$. 
Note also that $p_1\circ \Delta =p_2 \circ \Delta=\Id_Y$. 
We define the equalizer $Z\coloneqq Y \times_{Y \times_K Y} X$ of $f$ and $g$ as in the following Cartesian diagram
$$
\begin{CD}
Z=Y \times_{Y \times_K Y} X  @>{i}>> X \\
@V{\pi}VV @VV{h}V \\
Y @>{\Delta}>> Y\times_K Y
\end{CD}
$$
Since $\Delta$ is a closed immersion, we deduce that $i$ is also a closed immersion. 
Observe that $f \circ i=p_1 \circ h \circ i= p_1 \circ \Delta \circ \pi= \Id_Y \circ \pi=\pi$. 
Similarly, $g \circ i=p_2 \circ h \circ i= p_2 \circ \Delta \circ \pi= \Id_Y \circ \pi=\pi$. 
Therefore, $f \circ i=g \circ i$. 
\par
Now let $s$ be a closed point of $X$. 
By Lemma~\ref{l:closed-points}, 
the point $s$ is the image in $X$ of some $x\in X(\overline{K})\subset X(L)$. 
As $f^{(L)}(x)=g^{(L)}(x)=y\in Y(L)$ by hypothesis, we have
\[
p_1\circ (h \circ x) =f\circ x=g \circ x= p_2 \circ (h \circ x).
\] 
On the other hand, 
\[
p_1\circ (\Delta \circ y)=\Id_Y \circ y= p_2 \circ (\Delta \circ y).
\] 
The universal property of $Y\times_K Y$ implies that $h \circ x= \Delta \circ y \colon \Spec(L)\to Y\times_K Y$. 
By the universal property of $Z=Y \times_{Y \times_K Y} X$, there exists $t\in Z(L)$ such that $i \circ t=x$ and $\Delta \circ t=y$. 
In particular, let $z$ be the image of $t$ in $Z$ then $i(z)=s$. 
Since $X$ is Jacobson, the set of closed points $s$ is dense in $X$. 
Thus $i(Z)$ is dense in $X$. 
As $i$ is a closed immersion, $i(Z)$ is closed and thus $i(Z)=X$, i.e. $i$ is surjective. 
Therefore, $i$ is a surjective closed immersion and thus an isomorphism by the fact that $X$ is reduced (cf. Proposition 2.4.2.(c) ~\cite{liu-alg-geom})). 
Since $f \circ i=g \circ i$, we can conclude that $f=g$. 
\end{proof}

\begin{lemma} 
\label{l:reduced} 
Let $K$ be a perfect field and let $X,Y$ be reduced $K$-algebraic varieties. 
Then the following hold:
\begin{enumerate}[label=(\roman*)]
\item
$X,Y$ are geometrically reduced.
\item
the fibered product $X\times_K Y$ is reduced. 
\end{enumerate}
\end{lemma}

\begin{proof}
Assertion (i) is a direct consequence of Proposition 5.49.(i)-(iv) in~\cite{gorz}. 
Assertion  (ii) is Proposition~5.49.(ii) in~\cite{gorz}.  
\end{proof}

We are now in position to prove Theorem~\ref{t:inverse-also-general}. 

\begin{proof}[Proof of Theorem~\ref{t:inverse-also-general}] 
As in the proof of Theorem~\ref{t:surj-ca-general}, we can suppose that $G$ is residually finite. 
Clearly  (b) implies (a). 
For the converse, let $L$ be a field extending $\overline{K}$ such that $\tau\coloneqq \tau^{(L)}$ is reversible. 
Let $A\coloneqq X(L)$ and let $\FF$ denote  the set of all finite index subgroups of $G$. 
 Recall (cf.~Section~\ref{s:construction}) that,
 for each $H \in \FF$, 
 we have a bijective map $\rho_H \colon A^{H \backslash G} \to \text{Fix}(H)$ and a $K$-scheme morphism of algebraic varieties $\widetilde{\tau_H} \colon X^{H \backslash G} \to X^{H \backslash G}$ such that
\begin{equation}
\label{e:conjugate-b}
\widetilde{\tau_H}^{(L)}=\rho_H^{-1}\circ \tau \vert_{\text{Fix}(H)} \circ \rho_H.
\end{equation}
\par
We claim that $\widetilde{\tau_H}$ is a $K$-scheme automorphism of $X^{H \backslash G}$ for any $H \in \FF$. 
As $\tau$ is injective, $\tau\vert_{\text{Fix}(H)}$ is injective. 
Therefore, 
\[
\widetilde{\tau_H}^{(L)}\colon X^{H \backslash G}(L) \to X^{H \backslash G}(L)
\]
 is injective by \eqref{e:conjugate-b}. 
As $\overline{K} \subset L$, we have $X(\overline{K}) \subset X(L)$ and thus 
$\widetilde{\tau_H}^{(\overline{K})} \colon X^{H \backslash G}(\overline{K}) \to X^{H \backslash G}(\overline{K})$ is injective. 
Observe now that 
\[
\widetilde{\tau_H}^{(\overline{K})}= (\widetilde{\tau_H} \times \Id_{\overline{K}})^{(\overline{K})}  \colon X_{\overline{K}} ^{H \backslash G}(\overline{K}) \to X_{\overline{K}} ^{H \backslash G}(\overline{K}).
\]
We deduce that $\widetilde{\tau_H}\times \Id_{\overline{K}} \colon X_{\overline{K}}^{H \backslash G} \to X_{\overline{K}}^{H \backslash G}$ is injective by Lemma \ref{l:mor-var-closed-points}.(i). 
Since $X$ is reduced and $K$ is a perfect field, it is geometrically reduced. 
Hence $X_{\overline{K}}$ is geometrically reduced. 
Thus, $X_{\overline{K}}^{H \backslash G}$ is reduced by Lemma \ref{l:reduced}. 
As $X$ is separated, $X_{\overline{K}}^{H \backslash G}$ is also a separated $\overline{K}$-algebraic variety. 
Hence, the $\overline{K}$-scheme endomorphism $\widetilde{\tau_H}\times \Id_{\overline{K}}$ is an automorphism by Theorem \ref{t:automorphism} as we are in characteristic $0$. 
By virtue of Lemma \ref{l:flat-descent}, we see that $\widetilde{\tau_H}$ is indeed a $K$-scheme automorphism. 
\par
Since $\tau$ is reversible, $\tau^{-1}$ has a finite memory set $N'= \{g_1,\dots,g_n\}$ and we can suppose that $g_1=1_G$. 
Since $G$ is residually finite, there exists $H\in \FF$ such that the classes $Hg_1,\dots,Hg_n$ are pairwise distinct in $H \backslash G$. 
We can find $g_{n+1},\dots,g_{m}\in G$ such that $H \backslash G=\{Hg_1,\dots, Hg_n, Hg_{n+1}, \dots, Hg_{m}\}$. 
Then $N \coloneqq \{g_1,\dots,g_m\}$  is a memory set of $\tau^{-1}$ containing $N'$. 
Note that $\varphi_{H} \colon X^{H\backslash G} \to X^{N}$ (cf. Section \ref{s:construction}) is then a canonical isomorphism. 
Let $p_1\colon X^{H \backslash G} \to X$ be the canonical projection to the copy of $X$ indexed by the class $H\in H\backslash G$. 
We claim that the morphism of $K$-schemes 
\[
h \coloneqq p_1 \circ (\widetilde{\tau_H})^{-1} \circ \varphi_H^{-1} \colon X^{N} \to X
\] 
induces the local defining map of $\tau^{-1}$. 
To simplify  notation, let $\varphi\coloneqq \varphi_H^{(L)}$ and $\lambda\coloneqq \widetilde{\tau_H}^{(L)}$. 
Let $c\in A^G$. Since the classes $Hg_1,\dots,Hg_m$ are pairwise disjoint, there exists $d\in \text{Fix}(H)$ such that $c\vert_{N}=d\vert_{N}$. 
Observe that $\rho_H^{-1}(d) =\varphi^{-1}(d\vert_N) \in A^{H\backslash G}$ and that $\tau^{-1}\vert_{\text{Fix}(H)} = \rho_H \circ \lambda^{-1} \circ \rho_H^{-1}$ by \eqref{e:conjugate-b}. 
Since $N$ is a memory set of $\tau^{-1}$, we have that
\begin{align*}
& \tau^{-1}(c)(1_G) = \tau^{-1}(d)(1_G)=\tau^{-1}\vert_{\text{Fix}(H)}(d)(1_G)\\
 & =  \rho_H \circ \lambda^{-1}  \circ \rho_H^{-1}(d)(1_G)= p_1^{(L)} \circ \lambda^{-1}   \circ \rho_H^{-1}(d)\\
 &=  p_1^{(L)} \circ \lambda^{-1} \circ \varphi^{-1}(d\vert_N) =  h^{(L)}(d \vert_N) = h^{(L)}(c\vert_N).
 \end{align*}
This shows that $\tau^{-1} \colon A^G \to A^G$ is a cellular automaton whose local defining map is induced by the morphism $h$ as claimed.
\par
For $H\in \FF$, it is clear from \eqref{e:conjugate-b} and the reversibility of $\tau$ 
that over $A^{H \backslash G}$, i.e. over $L$-points, 
the $K$-morphism $(\widetilde{\tau_H})^{-1}$ must be the same as the $K$-morphism $\widetilde{(\tau^{-1})_H}$ 
constructed in Section \ref{s:construction} for the cellular automaton $\tau^{-1} \colon A^G \to A^G $ with local defining map induced by $h$. 
They are then equal as morphisms of schemes by Lemma \ref{l:equalizer} since $X^{H \backslash G}$ 
is separated and reduced by Lemma \ref{l:reduced} and that $L \supset \overline{K}$.  
In particular, they are the same over $Y$-points where $Y$ is any $K$-scheme. 
For each $K$-scheme $Y$, let $\sigma^{(Y)} \colon X(Y)^G \to X(Y)^G$ be the cellular automaton induced by the  morphism of $K$-schemes $h \colon X^{N} \to X$. 
Consider the following composition of cellular automata
\[
\tau^{(Y)} \circ \sigma^{(Y)} \colon X(Y)^G \to X(Y)^G.
\]
We have just shown that $\tau^{(Y)} \circ \sigma^{(Y)}$ and $\sigma^{(Y)} \circ \tau^{(Y)}$ are identities over $\text{Fix}(H)\subset X(Y)^G$ for each $H\in \FF$ by taking into account  \eqref{e:conjugate-b}.
Let $E \subset X(Y)^G$ denote the set of periodic configurations then $E=\bigcup_{H\in \FF}\text{Fix}(H)$. 
Since $G$ is residually finite, $E$ is dense in $X(Y)^G$. 
Therefore, $\tau^{(Y)} \circ \sigma^{(Y)}$ and $\sigma^{(Y)} \circ \tau^{(Y)}$ coincide with the identity map  on the dense subset $E \subset X(Y)^G$. 
As the prodiscrete topology on $X(Y)^G$ is Hausdorff and the maps
$\tau^{(Y)}, \sigma^{(Y)} $ are continuous since they are cellular automata, 
we concllude that  
$\tau^{(Y)} \circ \sigma^{(Y)}=\sigma^{(Y)} \circ \tau^{(Y)}=\Id$ on $X(Y)^G$.
Thus $\tau^{(Y)}$ is reversible with inverse $\sigma^{(Y)}$. 
This completes the proof of Theorem \ref{t:inverse-also-general}.
\end{proof}

\section{Counter-examples}

The following example (c.f.~\cite[Example~5.1]{cc-algebraic-ca}) 
shows that 
Theorem~\ref{t:closed-image} becomes false if we remove the hypothesis that $K$ is algebraically closed.

\begin{example} 
Let $G\coloneqq \Z$ and  $S\coloneqq\Spec(\R)$. Let $X=\Spec(\R[t])=\A^1_{\R}$ denote 
the affine line over $\R$. 
We also take  $Y_1\coloneqq \Spec(\R)$ and $Y_2\coloneqq \Spec(\C)$.
Let 
\[
A_1 \coloneqq X(\R)=\R \subset A_2 \coloneqq X(\C)=X_{\C}(\C)=\C.
\]   
For $i\in \{1,2\}$, let $\tau_i\colon A_i^G\to A_i^G$ be the cellular automaton defined over the schemes $S$, $X$, $Y_i$ with memory set  
$M \coloneqq \{0,1\}\subset G$ and local defining map $\mu_i \colon A_i^2 \to A_i$  
induced by the morphism of $S$-schemes $f\colon X^2 \to X$ given by the morphism of $\R$-algebras
\begin{align*}
\R[t] & \to \R[t_0,t_1] \\
  t & \mapsto  t_1-t^2_0. 
\end{align*}

We will show that $\tau_1(A_1^G)\subset A_1^G$ is dense for the prodiscrete topology but it is a proper subset of $A_1^{G}$. 
Consequently, the cellular automaton $\tau_1$ does not have the closed image property. 
Note that $\tau_1\colon \R^{\Z}\to \R^{\Z}$ is given by
\[
\tau_1(c)(n)=c(n+1)-c(n)^2 \quad \text{for all } c\in \R^{\Z}.
\]
Let us first show that $\tau_1(\R^{\Z})$ is  dense in $\R^{\Z}$. 
Let $d\in \R^{\Z}$ and let $F$ be any finite subset of $\Z$. 
Choose an integer $m \in \Z$ such that $F\subset [m,\infty)$. 
We define a configuration $c_F\in \R^{\Z}$ inductively by $c_F(n) \coloneqq 0$ if $n\leq m$ and $c_F(n+1) \coloneqq d(n )+c_F(n)^2$ for $n\geq m$. 
Clearly, $\tau(c_F)(n) = d(n)$ for all $n\geq m$, so that  $\tau(c_F)\vert_{F}=d\vert_{F}$. 
Hence $d$ is in the closure of $\tau_1(\R^{\Z})$.
This show that $\tau_1(\R^{\Z})$ is  dense in $\R^{\Z}$.
\par
Let us check now that the configuration $e\in \R^{\Z}$, defined by $e(n) \coloneqq 1$ for all $n\in \Z$, 
 does not belong to $\tau_1(\R^{\Z})$. 
 Suppose on the contrary that $\tau_1(c)=e$ for some $c\in \R^{\Z}$. 
This means that $c(n+1)=c(n)^2+1$ for all $n \in \Z$, 
so that  $1\leq c(n)$ and $c(n)<c(n+1)$ for all $n\in \Z$. 
Therefore, $c(n)$ has a finite limit as $n \to - \infty$.
However, this is impossible since the equation $t=t^2 + 1$ has no real roots. 
This shows that $e\notin \tau_1(\R^{\Z})$ and therefore $\tau_1(\R^{\Z})$ is not closed in $\R^{\Z}$.
\par
On the other hand, the cellular automaton $\tau_2$ has the closed image property by Theorem \ref{t:closed-image}. 
Consider the morphism of $\C$-schemes $f_{\C} \coloneqq f\times \Id_{\C} \colon X_{\C}^2 \to X_{\C}$.  
As in Remark \ref{rem:ca-base-change}, $\tau_2 \colon \C^{\Z}\to \C^{\Z}$ is also the cellular automaton over the schemes $\Spec(\C)$, $X_{\C}$, $\Spec(\C)$ with the same memory set $M$ and the local defining map induced by $f_{\C}$. 
Note that $\tau_2$ is given by the same formula
\[
\tau_2(c)(n)=c(n+1)-c(n)^2 \quad \text{for all } c\in \C^{\Z}.
\]
Hence, the same argument for $\tau_1$ shows that $\tau_2(\C^{\Z})$ is dense in $\C^{\Z}$. 
As $\tau_2(\C^{\Z})$ is also closed, $\tau_2$ is surjective. 
In fact, we check directly that $\tau$ is surjective and thus has the closed image property as follows. 
Let $d\in \C^{\Z}$. 
Define $c\in \C^{\Z}$ inductively by $c(0) \coloneqq 0$, $c(n+1)=d(n)+c(n)^2$ if $n\geq 0$ and $c(n) \coloneqq \sqrt{c(n+1)-c(n)}$ if $n\leq -1$. 
Then clearly $\tau(c)=d$. 
\par
\end{example}

The following  example shows that 
Theorem \ref{t:surj-ca} becomes false if we remove the hypothesis that $K$ is algebraically closed.

\begin{example}
Let $G$ be a group. 
Let $X \coloneqq \Proj^1_{\Q}$ denote  the projective line over $\Q$ 
(cf.~Example~2.3.34 in ~\cite{liu-alg-geom}).
Note that $X$ is a complete algebraic variety over $\Q$. 
Consider the morphism of $\Q$-schemes 
$f \colon \Proj_{\Q}^1 \to \Proj_{\Q}^1$ given by $(x \colon y) \mapsto (x^3 \colon y^3)$ for all $(x\colon y) \in \Proj^1$. 
Then $f^{(\Q)}$ is given by $(x\colon y) \mapsto (x^3\colon y^3)$ for all $(x\colon y) \in A\coloneqq \Proj_{\Q}^1(\Q)$. 
The map  $f^{(\Q)}$ is clearly  injective. 
However, it is not surjective since $(2\colon 1)\in A$ is not in the image of $f^{(\Q)}$. 
\par
Now let $\tau\colon A^G \to A^G$ denote the cellular automaton over the group $G$
and the schemes $\Spec(\Q)$, $X$, $\Spec(\Q)$,
 with memory set $M \coloneqq \{1_G\}$ and associated  local defining map $\mu \colon A \to A$ 
 induced by $f$. 
Then   $\tau$ is injective but not surjective.
\end{example}

The following example shows that the hypothesis that $X$ is of finite type cannot be removed
from Theorem~\ref{t:surj-ca}.

\begin{example}
Take $A := \N$ and let $G$ be any group.
Consider the cellular automaton $\tau \colon A^G \to A^G$ with memory set $M = \{1_G\}$
and local defining map $\mu_M \colon A^M = \N \to \N = A$ defined by $n \mapsto n + 1$.
Clearly $\tau$ is injective but not surjective.
However, as explained in the Introduction, given any field $K$ with algebraic closure $\overline{K}$,
$\tau$ can be regarded as a cellular automaton over the schemes
$S \coloneqq  \Spec(K)$, 
$X$, $Y \coloneqq \Spec(\overline{K})$,
where $X$ is the direct union of a family of copies of $\Spec(K)$ indexed by $A$.
\end{example}

In the next three examples,  $G\coloneqq \Z$ and  $K$ is a field. 
Let $S \coloneqq \Spec(K[t])$ and $X \coloneqq \Spec(K[t,u])$ 
be respectively the affine line and the affine plane over $K$. 
Consider any affine $S$-scheme $Y\coloneqq \Spec(R)$ where $R$ is a $K[t]$-algebra. 
The alphabet $A \coloneqq X_S(Y)$ can be identified with $R$. 
Indeed, 
\[
A=\Mor_{K[t]}(K[t,u],R)=R.
\]
Let $\tau\colon A^G\to A^G$ be the cellular automaton over the schemes $S,X,Y$ with memory set  
$M \coloneqq \{0,1\}\subset G$ and local defining map $\mu \colon A^2 \to A$  
induced by the morphism of $S$-schemes 
$f\colon X_S^2\to X$ given by the morphism of $K[t]$-algebras
\begin{align*}
K[t,u] = K[t][u]&\to K[t][u_0,u_1] \\
  u&\mapsto u_0 - tu_1.
\end{align*}
Clearly  $\tau(c)(n)=c(n)-tc(n+1)$ for all $c\in A^G$.
Note that $\tau$ is a group endomorphism of the additive group $A^G$.   
\par
The next example (c.f.~\cite[Example~1.10.3]{ca-and-groups-springer}) shows that a bijective cellular automaton over schemes $S,X,Y$ may fail to be reversible even if $S$ is a Noetherian Jacobson scheme, $X$ is of finite type over $S$ and $Y$ is Noetherian. 
\begin{example}
Let $R=K[[t]]$ denotes the $K$-algebra of formal power series in $t$.
Let $\sigma\colon A^G \to A^G$ be the map defined by 
\begin{equation}
\label{e:def-sigma}
\sigma(c)(n)=c(n)+tc(n+1)+t^2c(n+2)+\dots\quad \text{for all  } c\in A^G, n \in \Z.
\end{equation}
One immediately checks that  $\tau \circ \sigma =\sigma \circ \tau= \Id_{A^G}$. 
Hence, $\tau$ is bijective with inverse map $\sigma$.
However, the cellular automaton $\tau$ is not reversible. 
Indeed, suppose by contradiction that $\sigma$ is a cellular automaton
with memory set  $F \subset \Z$. 
Since $F$ is a finite set, there exists  an integer $m\geq 0$ such that $F\subset (-\infty,m]$. 
Consider the configurations  $c,d\in A^G$ such that $c(n) \coloneqq 0$ if $n\leq m$, $c(n)\coloneqq 1$ if $n>m$ and $d(n) \coloneqq 0$ for all $n\in \Z$. 
Since $c$ and $d$  coincide on $F$, we must have that $\sigma(c)(0) =\sigma(d)(0)$. 
But it follows from~\eqref{e:def-sigma} that $\sigma(c)(0)=t^{m+1}+t^{m+2}+\dots$ while 
$\sigma(d)(0)=0$. 
This  shows that $\sigma$ is not a cellular automaton and thus $\tau$ is not reversible.
\end{example}

The next example shows that in Theorem~\ref{t:surj-ca}, the surjunctivity property may fail to hold if the base scheme is not the spectrum of a field, even if 
the schemes $S,X,Y$ are all Noetherian Jacobson  schemes
and both $X$ and $Y$ are of finite type over $S$.

\begin{example}
Let $R=K[t]$. 
We claim that in this case, the automaton $\tau$ is injective but not surjective.
\par 
Indeed, suppose that $c \in (K[t])^{\Z}$ is in the kernel of $\tau$. 
This means that $c(n) = tc(n+1)$ for all $n\in \Z$.
As $c(n)$ is a polynomial in $t$ for all $n \in \Z$, this clearly implies    $c = 0$.
Therefore     $\tau$ is injective.
\par 
Now consider the configuration $d\in (K[t])^{\Z}$ such that $d(n)=1$ for all $n\in \Z$. 
Suppose that there is $c\in (K[t])^{\Z}$ such that $\tau(c)=d$. 
Then $c(n)-tc(n+1)=1$ for all $n\in \Z$. 
Hence for all $n\geq -1$, we have
\[
c(0)=1+tc(1)=1+t+t^2c(2)=...=1+t+ \dots +t^n+t^{n+1}c(n+1).
\]
As $c(n+1)\in K[t]$ for all $n\in \N$, this formula implies that $\deg(c(0))\geq n$ for all $n\in \N$, which is impossible.
Thus there is no $c\in K[t]^{\Z}$ such that $\tau(c)=d$.
This shows that  $\tau$ is not surjective.
\end{example}

\begin{example}
\label{e:surj}
Let  $\overline{K(t)}$ denote an algebraic closure of the field of fractions $K(t)$. 
Let  us take now $R \coloneqq \overline{K[t]}$,   the integral closure of $K[t]$ in $\overline{K(t)}$. 
Let us  show that, in this case also, the cellular automaton $\tau$ is injective but not surjective. 
\par
Suppose that the configuration $c\in (\overline{K[t]})^{\Z}$ is in the kernel of $\tau$. 
Then  $c(n) = tc(n+1)$ for all $n\in \Z$. 
It follows that $t^{-n} c(0)=c(n)\in \overline{K[t]}$ for all $n\in \Z$.
\par 
Let us establish the following: 
\begin{equation*}
\text{if } u\in \overline{K[t]}  \text{ satisfies } 
 t^{-n}u \in \overline{K[t]} \text{ for all } n\geq 0
 \text{ then } u = 0.
 \tag{$*$}   
\end{equation*}
Applying $(*)$  for $u=c(0)$, this will give us   that $c(0) = 0$ and hence $c(n)=t^{-n}c(0)=0$ for all $n\in \Z$, 
showing that $c = 0$ and hence that  $\tau$ is injective.
\par
To prove $(\star)$, 
suppose on the contrary that $u$ is a non-zero element of $ \overline{K[t]}$ that 
satisfies  $t^{-n}u \in \overline{K[t]} \text{ for all } n\geq 0$.
Since $u$ is integral over $K[t]$ and $u \not= 0$, there exist an integer $m\geq 1$ and $f_0,\dots,f_{m-1}\in K[t]$ such that
$f_0 \not= 0$ and 
\begin{equation}
\label{e:integral}
u^m+f_{m-1}u^{m-1}+\dots+ f_1u+f_0=0.
\end{equation}
Write $f_0=a_kt^k+\dots+a_0$ with $a_i \in K$ for $0 \leq i \leq k$.   
As $t^{-n}u \in \overline{K[t]}$ for all $n\geq 0$, 
we have that $f_iu^i t^{-n}\in \overline{K[t]}$ for all $n\geq 0$ and $i\ge 1$. 
Hence, by dividing both sides of ~\eqref{e:integral} by $t^n$, we see that for all $n\geq 0$,
\[
f_0t^{-n}=-(u^mt^{-n}+\dots+f_1ut^{-n})\in \overline{K[t]}.
\] 
Let $ s\in \{0,\dots, k\}$ be the minimal integer such that $a_{s}\in K^*$. 
Taking $n \coloneqq s+1>0$, we see that $a_st^{-1}\in  \overline{K[t]}$ since $a_kt^{k-n}+\dots+a_lt^{l-n}=f_0t^{-n}\in \overline{K[t]}$ and $a_st^{s-n}\in K[t]$ for all $s\geq n$. 
Thus, there exist an integer $p\geq 1$ and $h_0,\dots,h_{p-1}\in K[t]$ such that
\[
(a_st^{-1})^p+h_{p-1}(a_st^{-1})^{p-1}+\dots+h_0=0.
\]
Hence $a_s=-t(h_{p-1}a_s+\dots+h_0t^p)$, which is impossible since $a_s\in K^*$ and $h_{p-1}a_s+\dots+h_0t^p\in K[t]$. This shows that $u=0$. 
\par
Now consider the configuration $d\in (\overline{K[t]})^{\Z}$ such that $d(n)=1$ for all $n\in \Z$. 
Suppose that there is $c\in (\overline{K[t]})^{\Z}$ such that $\tau(c)=d$. 
Then  $c(n)-tc(n+1)=1$ for all $n\in \Z$. 
As in Example \ref{e:surj}, we have the following formula for all $n\geq 0$
\[
c(0)=1+t+ \dots +t^{n-1}+t^nc(n)=\frac{t^n-1}{t-1}+t^nc(n). 
\]
Therefore, we obtain for all $n\geq 0$
\[
(1+(t-1)c(0))t^{-n}=(t-1)c(n)+1\in   \overline{K[t]}.
\]
 As $1+(t-1)c(0)\in \overline{K[t]}$, we deduce from the general property $(*)$ that $1+(t-1)c(0)=0$, hence $c(0)=(1-t)^{-1}$. Observe that since $K[t]=K[1-t]\subset K(t)\subset \overline{K(t)}$, we have $\overline{K[t]}=\overline{K[1-t]}$. Hence $(*)$ is equivalent to 
 \begin{align*}
  \text{ if } u\in \overline{K[t]} \text{ and } (1-t)^{-n}u\in \overline{K[t]} \text{ for all } n\geq 0  \text{ then } u=0.
  \tag{$**$}  
\end{align*}
Since $\overline{K[t]}$ is a ring and $(1-t)^{-1}\neq 0$, we deduce from  $(**)$  that $c(0)=(1-t)^{-1}\notin \overline{K[t]}$. This contradiction shows that there is no $c\in (\overline{K[t]})^{\Z}$ such that $\tau(c)=d$. Therefore  $\tau$ is not surjective.
\end{example}

The following example shows that we can remove neither the hypothesis that $K$ is of characteristic $0$ nor the hypothesis that $K$ is algebraically closed
in Theorem~\ref{t:inverse-also}. 

\begin{example}
Let $G$ be a group and let $p$ be an odd prime number. 
Let $\overline{\F_p}$ be an algebraic closure of the finite field $\F_p = \Z/p\Z$. 
In this example, we take  $K=\R$ or $K=\overline{\F}_p$. 
Let $S=\Spec(K)$ and let $X = \Proj^1_{K}$ be the projective line over $K$ (cf.~Example~2.3.34 in ~\cite{liu-alg-geom}).   
Note that $X$ is a separated and reduced $K$-algebraic variety. 
Consider the morphism of $K$-schemes 
$f \colon \Proj^1_{K} \to \Proj^1_{K}$ given by $(x \colon y) \mapsto (x^p \colon y^p)$. 
Then $f^{(K)}$ is given by $(x\colon y) \mapsto (x^p\colon y^p)$ for all $(x\colon y) \in A\coloneqq \Proj_{K}^1(K)$ with $x, y \in K$. 
Since $K = \R$ or $K=\overline{\F_p}$ with $p$ odd, each $x\in K$ has a unique $p$-th root $\sqrt[p]{x}\in K$. 
We deduce that the map $f^{(K)}$ is bijective. 
\par
Let $\tau \colon A^G \to A^G$ denote the cellular automaton over $G$
and the schemes $S$, $X$, $S$,
 with memory set $M \coloneqq \{1_G\}$ and associated  local defining map $\mu \colon A \to A$ 
 induced by $f$. 
Then $\tau$ is reversible. 
The minimal memory set of the inverse $\tau^{-1}$ is $M=\{1_G\}$ and the associated local defining map $\nu \colon A \to A$ is given by $(x\colon y)\mapsto (\sqrt[p]{x} \colon \sqrt[p]{y})$ for each $(x\colon y) \in A$ with $x,y \in K$. 
\par
Suppose that $\tau^{-1}\in \CA(G,S,X,S)$. 
Then by definition, there exist a finite subset $N$ of $G$ and a morphism of $K$-schemes $h \colon X^N \to X$ such that $h^{(K)}$ is the local defining map associated with $N$. 
As $(1\colon 0) \in X(K)$ and $M=\{1_G\}$ is a memory set of $\tau^{-1}$, Proposition \ref{p:independent-of-memory} says that $\nu \colon A \to A$ is induced by a morphism of $K$-schemes $h' \colon X \to X$
and thus $h'^{(K)}=\nu$. 
Note that there exist nonzero homogeneous polynomials of the same degree $u,v \in K[x,y]$ such that $h'\colon \Proj^1_{K} \to \Proj^1_{K}$ is given by $(x\colon y)\mapsto (u(x,y) \colon v(x,y))$. 
Let $a(t)=u(t,1)$ and $b(t)=v(t,1)$ be nonzero polynomials in $K[t]$. 
Consider $w(t)=a(t)/b(t)\in K(t)$ then we have $h'((t\colon 1))=(w(t)\colon 1)$ for all $t\in K$ which are not zeros of $b(t)$. 
As $h'((t\colon 1))=\nu((t\colon 1))=(\sqrt[p]{t} \colon 1)$, we deduce that $\sqrt[p]{t}=w(t)$ for almost all $t\in K$. 
Therefore $a(t)=b(t)\sqrt[p]{t}$ and thus $a(t)^p=tb(t)^p$ for almost all $t\in K$. 
As $K$ is infinite, we deduce that $a(t)^p=tb(t)^p$ as polynomials in $K[t]$. 
Hence, $p\text{deg}(a)=1+p\text{deg}(b)$, which is a contradiction since $\text{deg}(a), \text{deg}(b)\in \N$ as $a,b$ are nonzero. 
We conclude that $\tau^{-1} \notin \CA(G,S,X,S)$, i.e., the inverse cellular automaton of $\tau$ is not defined over the schemes $S$, $X$, $S$. 
\end{example}

\appendix
\section{Schemes and algebraic varieties}
\label{sec:schemes}

In this section, we have collected the basic facts on schemes and algebraic varieties that are needed in the present paper. For more details, the reader is referred to 
\cite{gorz},
\cite{grothendieck-ega-1},
\cite{grothendieck-20-1964}, 
\cite{ega-4-3}, \cite{harris}, \cite{hartshorne}, 
\cite{liu-alg-geom}, \cite{milne}, and \cite{vakil}.
Some proofs are given in the case we have been unable to find a precise reference in the literature.
\par 
All rings are commutative with $1$ and
all ring morphisms are asked to send $1$ to $1$.

\subsection{Topological background}
Let $X$ be a topological space. If $Y$ is a subset of $X$, we denote by $\overline{Y}$ the closure of $Y$ in $X$.
\par
One says that  $X$ is  \emph{quasi-compact} if every open cover of $X$ admits a finite open subcover. Every closed subset of a quasi-compact topological space is itself quasi-compact for the induced topology.
\par
One says that  $X$ is  $T_0$, or \emph{Kolmogorov}, if, given any two distinct points of $X$, there is an open subset of $X$ that contains one of the two points but not the other.
This amounts to saying that, given any two distinct points $x$ and $y$ of $X$,
one has $x \notin \overline{\{y\}}$ or $y \notin \overline{\{x\}}$. 
\par
One says that a point $x\in X$ \emph{specializes} to a point $y$ of $X$ if $y \in \overline{\{x\}}$. 
A point $\xi \in X$ is called a \emph{generic point}  of $X$ if $\xi$ is the only point of $X$ that specializes to $\xi$.
\par
One says that $X$ is \emph{irreducible} if every non-empty open subset of $X$ is dense in $X$.
\par
A point $x \in X$ is called a \emph{closed point} of $X$  if the singleton $\{x\}$ is a closed subset of  $X$.  Note that, by Zorn's lemma,  every non-empty quasi-compact topological space admits a minimal non-empty closed (and hence quasi-compact) subset. It follows that any non-empty quasi-compact $T_0$ topological space
admits a closed point.
\par
 A subset of $X$  is said to be \emph{locally closed} if it is the intersection of an open subset and a closed subset of $X$.
 A subset  of  $X$ is said to be  \emph{constructible} if it is a finite union of locally closed subsets of $X$.

\begin{lemma}
 \label{l:const-open-dense-clo}
 Let $C$ be a constructible subset of a topological space $X$.
  Then there is an open dense subset $U$ of $\overline{C}$ such that $U \subset C$.
\end{lemma}

 \begin{proof}
 See~\cite[Lemma~2.1]{an-rigid}.
 \end{proof}

A subset $Y \subset X$ is said to be \emph{very dense} in $X$ if
$F \cap Y$ is dense in $F$  for every closed subset $F$ of $X$.
Note that if $Y$ is a very dense subset of $X$ then $Y$ must contain all the closed points of $X$.
One says that  $X$ is  \emph{Jacobson} if the closed points of $X$ are very dense
in $X$.
Note that $X$ is Jacobson if and only if every non-empty constructible subset of $X$ contains a closed point of $X$.
\par
The following result is a particular case of Proposition~10.3.2
 in~\cite{ega-4-3}. 

\begin{lemma}
\label{l:constructible-in-Jacobson}
Let $X$ be a Jacobson space and let $C$ be a constructible subset of $X$.
Then $C$ is a Jacobson space for the topology induced by $X$.
Moreover, 
a point of $C$ is closed in $C$ for the induced topology if and only if it is closed in $X$.
\end{lemma}

\begin{proof}
For  $Y \subset X$, let us denote by $Y_0$ the set of closed points of $Y$ (for the topology on $Y$ induced by $X$).
Observe that we always have the inclusion 
\begin{equation}
\label{e:closed-points-sub}
X_0 \cap Y \subset Y_0.
\end{equation}
Let $F$ be a closed subset of $C$ and let 
$U$ be an open subset of $C$ such that $F \cap U \not= \varnothing$.
To prove that $C$ is Jacobson, it suffices to show that $C_0$ meets 
$F \cap U$.
Since $C$ is constructible,
we can write $C$ as a finite union 
$C = \bigcup_{i \in I} G_i \cap V_i$, where 
 $G_i \subset X$ (resp.~$V_i \subset X$)  are  closed 
(resp.~open) subsets of $X$.
On the other hand, we have that
$F = C \cap H$ and $U = C \cap W$,
where $H \subset X$ (resp.~$W \subset X$) is a closed (resp.~open) subset of  $X$.
We then have
\[
F \cap U = \bigcup_{i \in I} H \cap G_i \cap V_i \cap W.
\]
As $F \cap U \not= \varnothing$, there exists $j \in I$ such  that 
$H \cap G_j \cap V_j \cap W \not= \varnothing$.
Now, $X_0 \cap H \cap G_j$ is dense in the closed subset $H \cap G_j \subset X$
since $X_0$ is very dense in $X$ by our hypothesis. 
As $V_j \cap W$ is open in $X$,
we deduce that $X_0$  meets $H \cap G_j \cap V_j \cap W$.
Therefore $X_0 \cap C$ is very dense in $C$.
This implies in particular that  $C_0 \subset X_0 \cap C$. 
Since $X_0 \cap C \subset C_0$ by~\eqref{e:closed-points-sub},
we conclude that $X_0 \cap C = C_0$ and that $C_0$ is very dense in $C$.
Note that the fact that the index set $I$ is finite is not used in this proof.  
\end{proof}

\begin{lemma}
\label{l:closure-point-Jacob}
Let $X$ be a Jacobson topological space and let $x$ be a non-closed point of $X$. 
Then  $\overline{\{x\}}$ is irreducible and 
 contains infinitely many closed points of $X$. 
\end{lemma}

\begin{proof}
Since $\{x\}$ is irreducible
and the closure of any irreducible subset of a topological space is itself irreducible,
it follows that $\overline{\{x\}}$ is irreducible.
\par
To prove that $\overline{\{x\}}$ contains infinitely many closed points of $X$, we proceed by contradiction.
Let $X_0$ denote the set  of closed points of $X$ and suppose that 
$F \coloneqq \overline{\{x\}} \cap X_0$ is finite.
Observe that $F$ is a  closed discrete subset of  $X$.
On the other hand $F$ is dense in $\overline{\{x\}}$ since $X$ is Jacobson. 
Thus $\overline{\{x\}} = F$,
which contradicts the fact that the point $x$ is not closed in $X$.
 \end{proof}

\begin{lemma}
\label{l:f-inj-closed-points}
Let $f \colon X \to Y$  be a continuous map between topological spaces.
Let $X_0$ (resp.~$Y_0$) denote the set of closed points of $X$ (resp.~$Y$).
Suppose that $X$ is Jacobson and that the restriction of $f$ to $X_0$ is injective.
Then one has $f^{-1}(Y_0)\subset X_0$.
\end{lemma}

\begin{proof}
We proceed by contradiction.
Suppose that there exist $x\in X \setminus X_0$ and $y \in Y_0$ such that $f(x)= y$.
As the point $x$ is not closed,
we deduce from Lemma~\ref{l:closure-point-Jacob} 
that there are distinct closed points $x',x''\in \overline{\{x\}}$.
We then have  $f(x')=f(x'')=y$ since $f(\overline{\{x\}})\subset \overline{f(x)}=\{y\}$. 
This contradicts the injectivity of the restriction of $f$ to $X_0$. 
\end{proof}

Let $f \colon X \to Y$ be a continuous map between topological spaces.
One says that $f$ is \emph{quasi-compact}
if the inverse image of every quasi-compact open subset of $Y$ is a quasi-compact subset of $X$.
One says that $f$ is \emph{dominant} if the image of $f$ is a dense subset of $Y$.
One says that $f$ has the \emph{closed image property} if the image of $f$ is closed in $Y$.
Note that $f$ is surjective if and only if it is dominant and has the closed image property.

\subsection{Schemes}
A \emph{ringed space} is a pair $(X,\OO_X)$, where $X$ is a topological space and $\OO_X$ is a sheaf of rings on $X$.
We will sometimes write $X$ instead of $(X,\OO_X)$. 
It will be clear from the context whether we are considering the scheme itself or its underlying topological space.
A ringed space $(X,\OO_X)$ is called a \emph{locally ringed space} if the stalk  $\OO_{X,x}$ of $\OO_X$ at each point $x \in X$ is a local ring.
If $(X,\OO_X)$ is a locally ringed space and $x \in X$, we shall denote by $\kappa_X(x)$ or simply 
$\kappa(x)$ the residue field of the local ring $\OO_{X,x}$. 
\par
The \emph{spectrum} of a ring $R$ is the locally ringed space $\Spec(R) = (\Spec(R),\OO_{\Spec(R)})$ defined as follows.
As a set, $\Spec(R)$ is the set of all prime ideals of $R$. 
The set   $\Spec(R)$ is equipped with the  \emph{Zariski topology}, i.e., the topology
whose closed sets are the sets $V(I)$ defined by
$V(I) \coloneqq \{\PP \in \Spec(R) : I \subset \PP\}$, where $I$ runs over all ideals of $R$.
The Zariski topology is quasi-compact. 
A base of open sets for the Zariski topology on $\Spec(R)$ is given by the sets
$D(f) \coloneqq \{\PP \in \Spec(R) : f \notin \PP \}$, where $f$ runs over $R$.
The structural sheaf $\OO_{\Spec(R)}$
is uniquely determined by the condition that  $\OO_{\Spec(R)}(D(f)) = R_f$ for all $f \in R$,
where $R_f$ denotes the localization of  $R$ by the powers of $f$.
\par
A locally ringed space is called an \emph{affine scheme} if it is isomorphic, as a locally ringed space, to the spectrum of some ring.
\par
An open subset $U \subset X$ of  a locally  ringed space $X = (X,\OO_X)$ is called \emph{affine}
if  the locally ringed space $(U,\OO_X\vert_U)$, obtained by restricting $\OO_X$ to $U$,  is an affine scheme.
A locally  ringed space $X$ is called a \emph{scheme}
if each point of $X$ admits an affine open neighborhood.
\par
The topological space underlying any scheme is $T_0$.
For the proof, we can restrict to affine schemes   since the property of being $T_0$ is clearly local.
Now let us assume that $X = \Spec(R)$ for some ring $R$.
If $\PP_1, \PP_2 \in X$ are distinct prime ideals of $R$, then, after possibly  exchanging $\PP_1$ and $\PP_2$, there exists $f \in R$ such that $f \notin \PP_1$ and $f \in \PP_2$.
We then have $\PP_1 \in \DD(f)$ and $\PP_2 \notin \DD(f)$.
As $\DD(f)$ is an open subset of $X$, this shows that $X$ is $T_0$.
\par
A scheme is said to be \emph{discrete} (resp.~\emph{quasi-compact}, resp.~\emph{irreducible}, resp.~\emph{Jacobson}) if its underlying 
topological space is discrete (resp.~quasi-compact, resp.~irreducible, resp.~Jacobson).

\begin{remark}
\label{rem:qc-subsets}
Observe that a scheme is quasi-compact if and only if it is a finite union of affine open subsets.
As $\Spec(R)$ is quasi-compact for any ring $R$, 
we deduce  that every subset of the  underlying topological space of a quasi-compact scheme is quasi-compact for the induced topology. 
\end{remark}

A scheme $X$ is called \emph{Noetherian} if the space $X$ admits a finite affine open cover 
$(U_i)_{i \in I}$ such that, for each $i \in I$,
one has $U_i = \Spec(R_i)$,
where $R_i$ is a Noetherian ring.
Every Noetherian scheme is quasi-compact.
\par 
A ring $R$ is called \emph{reduced} if it admits no non-zero  nilpotent elements. 
A scheme $X$ is said to be \emph{reduced} if  the local ring $\OO_{X,x}$ is reduced
for every $x \in X$.

\subsection{Scheme morphisms}
A \emph{scheme morphism} from a scheme $X$ to a scheme $Y$ is a locally ringed space morphism  from $(X,\OO_X)$ to $(Y,\OO_Y)$.
The \emph{category of schemes} is the category $\Sch$ whose objects are schemes and whose morphisms are scheme morphisms.
\par
If $X = \Spec(A)$ and $Y = \Spec(B)$ are affine schemes, then there is a canonical bijection from the set of scheme morphisms $X \to Y$  to the set of ring morphisms $B \to A$.
\par
Given a  scheme $S$, a scheme \emph{over} $S$, or an $S$-\emph{scheme}, is a scheme $X$ endowed with a scheme morphism 
$\pi \colon X \to S$. 
The morphism $\pi$ is called the \emph{structure morphism} of the $S$-scheme $X$. 
If $X$ and $Y$ are $S$-schemes with respective structure morphisms $\pi \colon X \to S$ and $\pi' \colon Y \to S$, a $S$-scheme morphism from $X$ to $Y$ is a scheme morphism
$f \colon X \to Y$ such that $\pi = f \circ \pi'$ in $\Sch$.
The \emph{category of $S$-schemes} is the category $\Sch_S$ whose objects are $S$-schemes and morphisms are $S$-scheme morphisms.
This category admits finite products
(see e.g.~\cite[Section~3.1.1]{liu-alg-geom}).
The product in $\Sch_S$ of a finite family of $S$-schemes is also called their 
$S$-\emph{fibered product}.
The $S$-fibered product of two $S$-schemes $X$ and $Y$ is denoted by $X \times_S Y$ or simply 
$X \times Y$ if there is no risk of confusion.
Similarly, if $X$ is an $S$-scheme and $E$ is a finite set,
the $S$-fibered product of a finite family of copies of $X$ idexed by $E$ is  simply denoted by  $X^E$. 

\begin{remark}
\label{rem:X-E-X-F}
Let $X$ be an $S$-scheme and let $E,F$ be finite sets.
Every map $\rho \colon E \to F$ induces an $S$-scheme morphism
$\rho^* \colon X^F \to X^E$ defined as follows.
Consider, for each $e \in E$, the projection $S$-scheme morphism $p_e \colon X^F \to X$ associated with the  index $\rho(e) \in F$.
Then $\rho^* \colon X^F \to X^E$
is the  $S$-scheme morphism associated, by the universal property of $S$-fibered products,  with the family of $S$-scheme morphisms
$p_e \colon X^F \to X$, $e \in E$.   
\end{remark}
 
If $X$ and $Y$ are two $S$-schemes,
 the set of $Y$-\emph{points} of $X$ is the set
$X(Y) \coloneqq \Mor_{\Sch_S}(Y,X)$ consisting of all $S$-scheme morphisms from $Y$ to $X$.

\begin{remark}
\label{rem:Z-points-fibered product}
Let $S$ be a scheme and let $X,Y,Z$ be $S$-schemes.
It is an immediate consequence of the universal property of  $S$-fibered products that there is a canonical set-theoretic identification
$(X \times_S Y)(Z) = X(Z) \times Y(Z)$.
It follows that if $E$ is a finite set, then there is a canonical identification
$X^E(Y) = (X(Y))^E$, where $X^M$ denotes as above the $S$-fibered product of a family of copies of $X$ indexed by $M$. 
\end{remark}

\begin{remark}
\label{rem:f-Y-points}
Let $S$ be a scheme and let $X$,$Y$,$Z$ be $S$-schemes.
Suppose that $f \colon X \to Z$ is an $S$-scheme morphism.
Then $f$ induces a map
$f^{(Y)} \colon X(Y) \to Z(Y)$ defined by $f^{(Y)}(\varphi) = f \circ \varphi$ for all $\varphi \in X(Y)$. 
\end{remark}

If $R$ is a ring, it is a common abuse to write $R$ instead of $\Spec(R)$.
In particular, a $\Spec(R)$-scheme is also called an $R$-scheme and we write $\Sch_R$ instead 
of $\Sch_{\Spec(R)}$ and $X \times_R Y$ instead of $X \times _{\Spec(R)} Y$.

\subsection{Base changes}
Let $S$ be a scheme and $T$ an $S$-scheme.
\par
If $X$ is an $S$-scheme, the $S$-fibered product $X \times T$ equipped with the second projection 
$X \times T \to T$ is a $T$-scheme, denoted by $X_T$.
One says that  $X_T$ is obtained from $X$ by the \emph{base change}
 $T \to S$.
\par
If $f \colon X \to Y$ is an $S$-scheme morphism between $S$-schemes,
one says that the $T$-scheme morphism $f_T \colon X_T \to Y_T$ defined as
the $S$-fibered product 
$f_T \coloneqq f \times \Id_T$ is obtained from $f$ by the \emph{base change} $T \to S$.

\begin{lemma}
\label{l:surj-base-change}
Let $S$ be a scheme and let $f \colon X \to Y$ be a morphism of $S$-schemes.
If $f$ is surjective, then $f_T \colon X_T \to Y_T$ is surjective for every base change $T \to S$.
\end{lemma}

\begin{proof}
See \cite[Proposition~3.5.2.(ii)]{grothendieck-ega-1}.
\end{proof}

\begin{remark}
\label{rem:T-points-X}
The set $X(T)$ of $S$-scheme morphisms $T \to X$ can be identified with the set $X_T(T)$ of 
$T$-scheme morphisms $T  \to X_T$. 
Indeed, by the universal property of fibered products, the map $f \mapsto (f,\Id_T)$ yields 
a canonical bijection   from  $X(T)$   onto  $X_T(T)$.
\end{remark}

\begin{remark}
\label{rem:ca-base-change}
Let $S$ be a scheme and let $X, Y$ be $S$-schemes.  Let $\tau \colon A^G\to A^G$ be a cellular automaton over the group $G$ and the scheme $X$  with coefficients in  $Y$. 
Thus, $A=X(Y)$ and the local defining map $\mu \colon A^M \to A$ is induced by an $S$-scheme morphism  $f\colon X^M \to X$. 
Let $f_Y\colon X_Y^M \to X_Y$ be the morphism of $Y$-schemes
obtained from $f$ by the base change associated with the structure morphism $Y \to S$. 
Note that $A=X(Y)=X_Y(Y)$ by the above remark. 
We claim that the local defining map $\mu$ is also induced by the morphism of $Y$-schemes $f_Y$. 
To see this, it suffices to consider, for any point $\alpha \in A^M=X^M(Y)=X_Y^M(Y)$, the following commutative diagram where the right square is Cartesian:
\[
\begin{CD}
Y   @> {\alpha}>> X_Y^M @>{f_Y}>> X_Y \\
@V{\text{Id}}VV  @VV{}V @V{}VV \\
Y @> {\alpha}>> X^M @>{f}>> X
\end{CD}
\]
As in the remark above, $f_Y \circ \alpha$ and $f \circ \alpha$ give the same point in $A^M$. 
Therefore, there will be no difference for $\tau \colon A^G\to A^G$ if we consider it as a cellular automaton over the $S$-scheme $X$ with coefficients in $Y$ or as a cellular automaton over the 
$Y$-scheme $X_Y$ with coefficients in $Y$. 
 \end{remark}

\subsection{Properties of  scheme morphisms}
Let $f \colon X \to Y$ be a morphism of schemes.
\par
One says that $f$ is
\emph{closed} (resp.~\emph{quasi-compact}, resp.~\emph{dominant}, resp.~\emph{injective}, resp.~\emph{surjective}, resp.~\emph{bijective}) 
if the map induced by $f$ between the topological spaces of the schemes is 
closed (resp.~quasi-compact, resp.~injective, resp.~surjective, resp.~bijective). 

\begin{definition}
\label{def:prop-scheme-mor}
Let $f \colon X \to Y$ be a morphism of schemes.
One says that $f$ is
\begin{itemize}
\item
a \emph{closed immersion} if for every affine open subset $V = \Spec(R)$ of $Y$, there is an ideal $I\subset R$ such that $f^{-1}(V)=\Spec(R/I)$ as schemes over $V=\Spec(R)$.
\item
\emph{locally of finite type} if for all $x \in $X, there exists an affine open neighbourhood $U = \Spec(B)$ of $x$ in $X$ and an affine open subset $V = \Spec(A)$ of $Y$ 
 with $f(U) \subset V$ such that the induced ring map $A\to B$ is of finite type
 (i.e., $B$ is finitely generated as an $A$-algebra).
\item
\emph{of finite type} if it is quasi-compact and locally of finite type.
\item
\emph{separated} if the diagonal morphism 
$\Delta_{X/Y} = \Id_X \times_Y f  \colon X \to X \times _Y Y$ is a closed immersion.
\item
\emph{proper} if it is of finite type, separated, and universally closed.
\end{itemize}
\end{definition}

\begin{theorem}[Chevalley's constructibility theorem]
\label{t:chevalley}
Let $X$ and $ Y$ be Noetherian schemes.  Let $f \colon X \to Y$ be a scheme morphism of finite type.
Then the image by $f$ of every constructible subset of $X$ is a constructible subset of $Y$.
In particular, $f(X)$ is a constructible subset.
\end{theorem}

\begin{proof}
See \cite[Th\'eor\`eme~1.8.4]{grothendieck-20-1964}, \cite[p.~93]{hartshorne},  
\cite[Theorem~7.4.2]{vakil}.
\end{proof}

\begin{corollary}
\label{c:image-dominant}
Let $f \colon X \to Y$ be a  scheme morphism of finite type between Noetherian schemes and let
$C$ be a constructible subset of $X$. 
Then there is a dense open subset of the closure of $f(C)$ in $Y$ that is contained in $f(C)$.  
\end{corollary}

\begin{proof}
This immediately follows from Theorem~\ref{t:chevalley} and Lemma~\ref{l:const-open-dense-clo}.
\end{proof}
Let $\PP$ be a property of scheme morphisms.
One says that  $\PP$  is \emph{universal} if $\PP$ is stable under base change, i.e.,
if $f \colon X \to Y$ is a morphism of $S$-schemes that satisfies $\PP$, then 
$f_T \colon X_T \to Y_T$ also satisfies $\PP$ for any base change $T \to S$.
All the properties of scheme morphisms listed in Definition~\ref{def:prop-scheme-mor} are universal and stable under composition of scheme morphisms.

\begin{lemma}
\label{l:comp-finite-type}
Let $f\colon X \to Y$ and $g\colon Y\to Z$ be scheme morphisms.
If the composite morphism  $g \circ f$   is of finite type and  $f$ is quasi-compact, 
then $f$ is of finite type.
\end{lemma}

\begin{proof}
See \cite[Proposition~3.2.4.(e)]{liu-alg-geom}.
 \end{proof}
\begin{lemma}
\label{l:ft-and-Jacobson}
Let $f \colon X \to Y$ be a schem morphism.
Suppose that $f$ is of finite type and $Y$ is Jacobson.
Then $X$ is Jacobson.
Moreover, if $X_0$ (resp.~$Y_0$) denotes the set of closed points of $X$ (resp.~$Y$),
then one has $f(X_0) \subset Y_0$.
\end{lemma}

\begin{proof}
See \cite[Corollaire~10.4.7]{ega-4-3}.
\end{proof}
 
Given a scheme $S$, a scheme $X$ over  $S$ is said to be 
 \emph{of finite type}  (resp.~\emph{separated}, resp.~\emph{proper}, resp.~\emph{affine})  
if the structure morphism $X \to S$ is 
of finite type (resp.~separated, resp.~proper, resp.~affine).  

\subsection{Algebraic varieties}
Let $K$ be a field. 
A scheme $X$ over $K$ is   called an \emph{algebraic variety} over $K$,
or a $K$-\emph{algebraic variety},
if $X$ is of finite type.
This amounts to saying that $X$ admits a finite cover $(U_I)_{i \in I}$ by affine open subsets such that
$U_i = \Spec(R_i)$ for some finitely generated $K$-algebra $R_i$ for each $i \in I$
(cf.~Definition~2.3.47 and Example~3.2.3 in ~\cite{liu-alg-geom}).
For instance, if $V \subset K^n$ is a $K$-algebraic subset, then its associated $K$-scheme, i.e., the spectrum of the coordinate ring $K[V]$ is a $K$-algebraic variety.
\par
A $K$-algebraic variety $X$  is said to be \emph{complete} if the $K$-scheme $X$  is proper.
Every \emph{projective} algebraic variety is complete but there are
complete algebraic varieties that are not projective
(cf. Definition~3.1.12 and Proposition~3.3.30 in ~\cite{liu-alg-geom}).
The $\C$-algebraic variety associated with a $\C$-algebraic subset $V \subset \C^n$ is complete if and only if $V$ is compact for the usual topology on $\C^n$.
\par
Let $\overline{K}$ be an algebraic closure of $K$.
A $K$-algebraic variety $X$  is said  to be \emph{geometrically reduced} if the scheme 
$X_{\overline{K}}$, obtained from $X$ by the base change
$\Spec(\overline{K}) \to \Spec(K)$, is reduced.

\begin{lemma}
\label{l:prop-alg-var}
Let $K$ be a field and let $X$ be an algebraic variety over $K$.
If $Y \subset X$, let $Y_0$ denote the set of closed points of $Y$ for the induced topology.
Then the following hold:
\begin{enumerate}[label=(\roman*)]
\item
$X$  is Noetherian;
\item
every subset of $X$ is quasi-compact for the induced topology;
\item
$X$ is  Jacobson;
\item
every constructible subset $C \subset X$ is Jacobson for the induced topology and satisfies 
$C_0 = C \cap X_0$.
\end{enumerate}
\end{lemma}

\begin{proof}
Assertion (i) immediately follows from the fact that every finitely generated $K$-algebra is Noetherian by Hilbert's theorem.
Assertion (ii) is an immediate consequence of (i) since any Noetherian scheme is quasi-compact
and every subset of the topological space underlying a quasi-compact scheme is itself quasi-compact
for the induced topology
(see Remark~\ref{rem:qc-subsets}).
\par
Assertion (iii) follows from Lemma~\ref{l:ft-and-Jacobson}
since $\Spec(K)$ is reduced to a point and hence  trivially Jacobson.
\par  
The last assertion is a consequence of (iii) by Lemma~\ref{l:constructible-in-Jacobson}.
\end{proof}

\begin{lemma}
\label{l:mor-var-finite-type}
Let $K$ be a field.
Then every $K$-scheme morphism between $K$-algebraic varieties is of finite type.
\end{lemma}

\begin{proof}
Let $f \colon X \to Y$ be a $K$-scheme morphism between $K$-algebraic varieties $X$ and $Y$.
The morphism $f$ is quasi-compact since any subset of $X$ is quasi-compact for the induced topology by Lemma~\ref{l:prop-alg-var}.(ii).
On the other hand, if $\pi \colon Y \to \Spec(K)$ is the structure morphism of $Y$,
then the composite morphism $\pi \circ f \colon X \to \Spec(K)$ is the structure morphism of $X$ and is therefore of finite type.
By applying Lemma~\ref{l:comp-finite-type}, 
we deduce that $f$ is of finite type. 
\end{proof}

\begin{lemma}
\label{l:closed-points}
Let $X$ be an algebraic variety over a field $K$ and let $x\in X$. 
Then the following conditions are equivalent:
\begin{enumerate}[label=(\alph*)]
\item
$x$ is a closed point of $X$;
\item
$\kappa(x)$ is a finite field extension of $K$;
\item
$\kappa(x)$ is an algebraic field extension of $K$.
\end{enumerate}
\end{lemma}

\begin{proof}
See \cite[Proposition~6.4.2]{grothendieck-ega-1} for the equivalence between (a) and (b). Clearly, (b) implies (c) since every finite field extension is algebraic.
\par
Conversely, 
suppose that $\kappa(x)$ is an algebraic extension of $K$. 
As the scheme $X$ is of finite type over $K$, the residual field $\kappa(x)$ is a 
finitely generated $K$-algebra.
 Since each of its generators is of finite degree over $K$,
we deduce that  $\kappa(x)$ is a finite extension of $K$.  
This shows that  (c) implies (b).
\end{proof}

\begin{lemma}
\label{l:closed-points-alg-closed-var}
Let $X$ be an algebraic variety over an algebraically closed field $K$.
Then the map from $X(K)$ into $X$, that sends each $f \in X(K)$ to the image by $f$ of the unique point of $\Spec(K)$,
yields a bijection from $X(K)$ onto the set of closed points of $X$.
\end{lemma}

\begin{proof}
See \cite[Corollaire~6.4.2]{grothendieck-ega-1}.
\end{proof}

\begin{remark}
\label{rem:closed-points-rat-points}
In the case when  $X$ is an algebraic variety over an algebraically closed field $K$,
Lemma~\ref{l:closed-points-alg-closed-var} allows us to identify $X(K)$ with the set of closed points of $X$.
\end{remark}

\begin{lemma}
\label{l:mor-var-closed-points}
Let $K$ be a field
 and let $f \colon X \to Y$ be a $K$-scheme morphism between $K$-algebraic varieties. 
Let $X_0$ (resp.~$Y_0$) denote the set of closed points of $X$ (resp.~$Y$).
Then one has
\begin{enumerate}[label=(\roman*)]
\item
 $f(X_0) \subset Y_0$.
\end{enumerate}
Moreover, if $f_0 \colon X_0 \to Y_0$ denotes the map obtained by restricting $f$ to $X_0$, then the following hold:
\begin{enumerate}[label=(\roman*),resume]
\item
if $f_0$ is injective, then $f^{-1}(Y_0) \subset X_0$;
\item
$f$ is injective if and only if $f_0$ is injective;
\item
$f$ is surjective if and only if $f_0$ is surjective;
\item
if $C$ is a constructible subset of $X$, then
\[
f_0(C \cap X_0) = f(C \cap X_0) = f(C) \cap Y_0;
\]
\item
if $D$ is any subset of $Y$, then
\[
f_0^{-1}(D \cap Y_0) = f^{-1}(D) \cap X_0; 
\]
\item
if the map $f \colon X \to Y$ is closed, then the map $f_0 \colon X_0 \to Y_0$ is also closed.
\end{enumerate}
\end{lemma}

\begin{proof}
(i).
Let $x \in X_0$ and $y \coloneqq f(x) \in Y$.
Since $f$ is a $K$-scheme morphism, we have field morphisms
$K \to \kappa(y) \to \kappa(x)$.
As the field $\kappa(x)$ is a finite extension of $K$ by Lemma~\ref{l:closed-points},
it follows that
$\kappa(y)$ is also a finite extension of $K$.
By applying again Lemma~\ref{l:closed-points}, we deduce that $y \in Y_0$. 
This shows 
that $f(X_0) \subset Y_0$.
\par
\noindent
(ii).
Suppose that  $f_0$ is injective.
Since $X$ is Jacobson by Lemma~\ref{l:prop-alg-var}.(iii),  
we have that $f^{-1}(Y_0) \subset X_0$
by Lemma~\ref{l:f-inj-closed-points}.
\par
\noindent
(iii).
The injectivity of  $f$ trivially implies that of  $f_0$.
\par
Suppose now that $f_0$ is injective but $f$ is not injective. 
Hence, there exist $x_1, x_2\in X$ and $y\in Y$ such that $x_1\neq x_2$ and $f(x_1)=f(x_2)=y$.
Since $X$ is Jacobson, we have that
$f^{-1}(Y_0) \subset X_0$ by Lemma~\ref{l:f-inj-closed-points}. 
Therefore  $y$ is not a closed point.
As $f(X_0) \subset Y_0$, it follows that   $x_1$ and $x_2$ are not closed either.
Now let $U \coloneqq  \overline{\{x_1\}}$, $V \coloneqq  \overline{\{x_2\}}$,
and $W \coloneqq  \overline{\{y\}}$.
First observe that $f(U) \subset W$ and $f(V) \subset W$ by continuity of $f$.
On the other hand, up to exchanging $x_1$ and $x_2$,
we can assume that $x_1 \notin V$ since $X$ is $T_0$.
If follows from  Corollary~\ref{c:image-dominant} that each of the sets
$f(U \setminus V)$ and $f(V)$ contains an open dense subset of $\overline{f(U \setminus V)} = \overline{f(V)} = W$.
As $W$ is irreducible by Lemma~\ref{l:closure-point-Jacob}, we deduce that $f(U \setminus V) \cap f(V)$ contains a non-empty open subset  
$T \subset W$. 
Now, by Lemma~\ref{l:constructible-in-Jacobson},
the open set $T \subset W$, which is locally closed in $Y$,  
contains a  point $y' \in Y_0$.
Let $x_1' \in U \setminus V$ and $x_2' \in V$  such that $f(x_1') = f(x_2') = y'$.
We have that $x_1',x_2' \in X_0$ by (ii). 
Since  $x_1'\neq x_2'$, this contradicts the injectivity of $f_0$  and therefore completes the proof of  (i).
\par
\noindent
(iv).
Suppose that  $f$ is not surjective. 
By Chevalley's theorem, $f(X)$ is constructible in $Y$. 
Thus $Y \setminus f(X)$ is constructible and non-empty. 
Since $Y$ is Jacobson, Lemma~\ref{l:constructible-in-Jacobson} implies that $Y\setminus f(X)$ contains a closed point $y \in Y_0$. Therefore $f_0$ is not surjective.
Thus  the surjectivity of $f_0$ implies that of $f$.
\par 
Conversely, suppose now that $f$ is surjective. 
Let $y \in Y_0$.
As $f$ is surjective,  there exists $x \in X$ such that $f(x) = y$.
Since $X$ is Jacobson, there is a closed point $z$ of $X$ in $\overline{\{x\}}$. 
Since $f(\overline{\{x\}}) \subset \overline{f(x)}=\{y\}$, we have that $f_0(z) = f(z)=y$.
This shows that $f_0$ is surjective.
\par
\noindent
(v). 
Let $C$ be a constructible subset of $X$. We have that
\[
f_0(C \cap X_0) = f(C \cap X_0) \subset  f(C) \cap Y_0
\]
by definition of $f_0$.
Now let $y\in f(C) \cap Y_0$. Then $y=f(x)$ for some $x\in C$. 
The set  $\overline{\{x\}}\cap C$ is a closed subset of $C$. 
Since $C$ is Jacobson and $C \cap X_0$ is the set of closed points of $C$ for the topology  induced on $C$ by $X$  
(cf.~Lemma~\ref{l:constructible-in-Jacobson}), 
there exists a  point $x'\in \overline{\{x\}}\cap C \cap X_0$.
Since $f$ is continuous and $y$ is a closed point of $Y$,
we have that  $f(x')\in f(\overline{x})=\{y\}$, so that  $f(x')=y$ 
and hence $y\in f(C \cap X_0)$. This shows that $f(C) \cap   Y_0\supset f(C \cap X_0)$ and thus $f(C \cap X_0)= f(C) \cap Y_0$. 
\par
\noindent
 (vi). Let $D$ be a subset of $Y$.
 We have the inclusion $f^{-1}(D) \cap X_0 \subset f_0^{-1}(D \cap Y_0)$ by definition of $f_0$.
 Now let $x \in f_0^{-1}(D \cap Y_0)$.
Then $x \in X_0$ and $f(x) \in D$.
Therefore $x \in f^{-1}(D) \cap X_0$.
This shows that $f_0^{-1}(D \cap Y_0) \subset f^{-1}(D) \cap X_0$ and hence
$f_0^{-1}(D \cap Y_0) = f^{-1}(D) \cap X_0$.
\par
\noindent
(vii).
By (iii), for every closed subset $F$ of $X$, we have that
\[
f_0(F \cap X_0) = f(F) \cap Y_0.
\]
Therefore $f_0$ is closed if $f$ is.
\end{proof}

\begin{remark}
Assertion (i) in the above lemma is a consequence of \cite[Corollaire~10.4.7]{ega-4-3}.
Assertions (iii) and (iv) are proved in \cite[Section~2]{nowak} in the more general setting of morphisms of finite type between Jacobson Noetherian schemes.
\end{remark}

The following result is a special case of what is known as \emph{faithfully flat descent}.

\begin{lemma}
\label{l:flat-descent}
Let $K$ be a field and let $f \colon X \to Y$ be a $K$-scheme morphism between $K$-algebraic varieties. 
Let $L/K$ be a field extension. 
Suppose that $f_L\coloneqq f\times \Id_L \colon X_L \to Y_L$ is an $L$-isomorphism. 
Then $f$ is a $K$-isomorphism.
\end{lemma}

\begin{proof}
See  \cite[Proposition~14.51.(2)]{gorz}. 
\end{proof}

\subsection{The Ax-Grothendieck theorem}
In the proof of Theorem~\ref{t:surj-ca},
we shall use the following result  
\cite{ax-injective}
(see also \cite[proposition~10.4.11]{ega-4-3}). 

\begin{theorem}[Ax-Grothendieck]
\label{t:surj-grothendieck}
Let $S$ be a scheme and let $X$ be a separated scheme of finite type over $S$.
Then every injective $S$-endomorphism of $X$ is surjective (and hence bijective).
\end{theorem}

The Ax-Grothendieck theorem  can be strengthened
for separated reduced algebraic varieties as follows. 

\begin{theorem}[Nowak]
\label{t:automorphism}
Let $K$ be an algebraically closed field of characteristic $0$. 
Let $X$ be a separated and reduced $K$-algebraic variety. 
Then every injective $K$-scheme endomorphism of $X$ is a $K$-scheme automorphism.
\end{theorem}

\begin{proof}
See~\cite{nowak}.
\end{proof}

\section{Inverse limits}
\label{sec:inverse}

In this section we review inverse limits and establish a result on the non-emptiness of inverse limits of constructible subsets of algebraic varieties.

\subsection{Inverse systems}
Let $I$ be a directed poset, i.e., a set $I$ with a partial ordering $\prec$ such that, for all $i,j \in I$, there exists $k \in I$ such that
$i \prec k$ and $j \prec k$. 
An \emph{inverse system} of sets  \emph{indexed} by  $I$ consists of the following data:
(1) a family of sets $(Z_i)_{i \in I}$ indexed by $I$;
(2) for each pair $i,j \in I$ such that $i \prec j$, a
\emph{transition map}
$\varphi_{ij} \colon Z_j \to Z_i$
 satisfying the following conditions:
\begin{align*}
 \varphi_{ii} &= \Id_{Z_i} \text{ (the identity map on $Z_i$) for all } i \in  I, \\[4pt]
 \varphi_{ij} \circ \varphi_{jk}  &= \varphi_{ik}  \text{ for all $i,j,k \in I$ such that } i \prec j \prec k. 
\end{align*}
One then speaks of the inverse system $(Z_i,\varphi_{ij})$, or simply of the inverse system $(Z_i)$ if
the transition maps  are clear from the context.
 \par
The \emph{inverse limit} of an inverse  system $(Z_i,\varphi_{ij})$
is the subset $Z \subset \Pi_{i \in I} Z_i$ defined by
\[
Z \coloneqq \{ (z_i)_{i \in I}  :  \varphi_{ij}(z_j) = z_i \text{ for all } i,j \in I \text{ such that } i \prec j \} 
\]
and one writes $Z= \varprojlim Z_i$. 

\subsection{Universal elements}
Let $(Z_i,\varphi_{i j})$ be an inverse system of sets indexed by the directed set $I$. 
For each $i \in I$, 
The  set
\[
Z_i' \coloneqq \bigcap_{i \prec j} \varphi_{i j}(Z_j) \subset Z_i.
\]
is called the set of \emph{universal elements} of $Z_i$
(cf. \cite[p.~408]{grothendieck-ega-3}).
Observe that $\varphi_{i j}(Z_j') \subset Z_i'$ for all $i \prec j$.
Consequently,   the map $\varphi_{i j}$ induces by restriction a map 
$\varphi_{i j}' \colon Z_j' \to Z_i'$. 
 The inverse system  $(Z_i',\varphi_{i j}')$ is called the
\emph{universal inverse system} associated with the inverse system $(Z_i,\varphi_{i j})$.
Note that the inverse systems $(Z_i,\varphi_{i j})$ and $(Z_i',\varphi_{i j}')$ have clearly the same inverse limit.

\subsection{Stone's theorem}
We shall make use of the following result which is  due to 
A.H.~Stone.

\begin{theorem}[Stone]
\label{t:stone} 
Let $(Z_i,\varphi_{ij})$ be an inverse system of non empty, quasi-compact, $T_0$ topological spaces $Z_i$  with continuous closed transition maps $\varphi_{ij} \colon Z_j \to Z_i$.  
Then one has $\varprojlim Z_i \not= \varnothing$.
\end{theorem}

\begin{proof}
See \cite[Theorem~3]{stone-1979}.
\end{proof}

\subsection{Inverse sequences}
An inverse system of sets whose index set is $\N$ is called an \emph{inverse sequence} of sets.
\par
Note that a sufficient  condition
for an inverse sequence $(Z_n,\varphi_{n m})$ to have a non-empty inverse limit is that
$Z_n$ is non-empty   for some $n \in \N$ and $\varphi_{n m}$ is surjective for all $m,n \in \N$ with $n \leq m$.

\subsection{Inverse limits of closed points of constructible subsets}
We shall also make use of the following result.

\begin{lemma} 
\label{l:inverse-limit-seq-const}
Let $K$ be an uncountable field and let $(X_n,f_{n m})$ be an inverse sequence such that  
 each  $X_n$ is a $K$-algebraic variety  and each transition map  
$f_{n m} \colon X_m \to X_n$ is a $K$-scheme morphism.
 Suppose that we are given, for each $n \in \N$, a non-empty constructible subset $C_n\subset X_n$  and that 
$f_{nm}(C_m)\subset C_n$ for all $m,n \in \N$ with $m\geq n$. 
Let $Z_n$ denote the set of closed points of $C_n$.
Then one has $f_{n m}(Z_m) \subset Z_n$ for all $n,m \in \N$ with $n \leq m$. 
Moreover, the inverse sequence $(Z_n,\varphi_{nm})$,
where $\varphi_{n m} \colon Z_m \to Z_n$ is the restriction of $f_{n m}$ to $Z_m$, 
verifies 
$\varprojlim Z_n \neq \varnothing$.
\end{lemma}

Let us first establish an auxiliary result.

\begin{lemma}
\label{l:nested-sequence}
Let $K$ be an uncountable field  with algebraic closure $\overline{K}$ and
let $X$ be an algebraic variety over $K$.  
 Suppose that $(C_n)_{n\in \N}$ is a sequence of non-empty constructible subsets of $X$ such that $C_{n+1}\subset C_n$ for all $n\in \N$.
Let $Z_n$ denote the set of closed points of $C_n$.  
 Then one has $\bigcap_{n\in \N}Z_n\neq \varnothing$.
\end{lemma}

\begin{proof} 
Since $X$ is an algebraic variety,
it admits a finite cover  $(X_i)_{i \in I}$ by affine open subsets. 
We claim that there exists $j \in I$ such that $X_j \cap C_n \neq \varnothing$ for all $n \in \N$. 
Indeed, otherwise, for each $i \in I$,there would exist $n_i \in \N$ such that 
$X_i \cap C_n=\varnothing$ for all $n\geq n_i$. Taking $m\coloneqq \max_{i \in I} n_i$, 
we would then get 
\[
C_m =X\cap C_m= \bigcup_{i \in I} X_i \cap C_m =\varnothing,
\]
which is a contradiction. 
Hence, we get a new decreasing sequence of non-empty constructible subsets of $Y \coloneqq X_j$ by setting $C'_n\coloneqq Y \cap C_n$ for all $n\in \N$. 
Consider the base change $\Spec(\overline{K})\to \Spec(K)$ and the canonical morphism
\[
g \colon Y_{\overline{K}}=Y\times_K \Spec(\overline{K}) \to Y. 
\]
 Let $C''_n=g^{-1}(C'_n)$ for all $n\in \N$. 
 Let $Z'_n$ (resp.~$Z''_n$) denote the  set of closed points of $C'_n$ (resp.~$C''_n$). 
 Clearly, $(Z''_n)_{n\in \N}$ is a non-increasing sequence of non-empty constructible subsets of the 
 affine $\overline{K}$-algebraic set $V$ consisting  of all closed points of $Y_{\overline{K}}$.
    Observe also that, for all $n \in \N$, we have that  $Z''_n \not= \varnothing$ since 
    $C''_n \not= \varnothing$ and $C_n''$ is Jacobson by Lemma~\ref{l:constructible-in-Jacobson}.
 By ~\cite[Proposition~4.4]{cc-algebraic-ca}, there exists $x\in \bigcap_{n\in \N} Z''_n$. 
 Thus $f(x) \in \bigcap_{n\in \N} Z'_n\subset \bigcap_{n\in \N} Z_n$. 
This shows that $\bigcap_{n\in \N}Z_n$ is not empty.
\end{proof}

We are now in position for proving Lemma~\ref{l:inverse-limit-seq-const}.

\begin{proof}[Proof of Lemma~\ref{l:inverse-limit-seq-const}]
As $f_{n m}$ is a scheme morphism of finite type by Lemma~\ref{l:mor-var-finite-type} and the schemes $X_n, X_m$ are Noetherian by Lemma~\ref{l:prop-alg-var}.(i),
Chevalley's theorem (cf.~Theorem~\ref{t:chevalley}) implies that 
the set  $f_{n m}(C_m)$ is a constructible subset of $X_n$ for all $n,m \in \N$ with $n \leq m$.   
Therefore, for each $n\in \N$, the sequence $(f_{n m}(C_m))_{n\leq m}$ is a non-increasing sequence of 
non-empty constructible subsets of  $X_n$.
On the other hand, $f_{n m}(Z_m)$ is the set of closed points of $f_{n m}(C_m)$ by 
Lemma~\ref{l:mor-var-closed-points}.(v). This shows in particular that  $f_{n m}(Z_m) \subset Z_n$.
Let $\varphi_{n m} \colon Z_m \to Z_n$ denote the restriction of $f_{n m}$ to $Z_m$
and let $(Z_n',\varphi_{n m}')$ denote the universal inverse sequence associated with the inverse sequence $(Z_n,\varphi_{n m})$.
By applying Lemma~\ref{l:nested-sequence}, we deduce  that 
\[
Z'_n= \bigcap_{n \leq m} \varphi_{nm}(Z_m) = \bigcap_{n \leq m} f_{nm}(Z_m) \neq \varnothing.
\]
Therefore, to prove that $\varprojlim Z_n \neq \varnothing$, it suffices to show that all transition maps 
$\varphi_{n m}' \colon Z_m' \to Z_n'$ are surjective.
\par
To see this, let $n \leq m$ and let $z_n' \in Z_n'$.
Then, for all $n \leq k$, we have that $z'_n \in \varphi_{n k} (Z_k)$ and thus there exists $y_k\in Z_k$ such that $z_n' = \varphi_{nk}(y_k)$.
For $m \leq k$, let $s_k \coloneqq \varphi_{m k}(y_k)$. 
We then have
\[
\varphi_{n m}(s_k) =\varphi_{n m}\circ \varphi_{mk}(y_k)=\varphi_{n k}(y_k)=z'_n.
\]
Therefore, for all $m \leq k$, 
the set $\varphi_{nm}^{-1}(z'_n) \cap \varphi_{mk}(Z_k)$ is non-empty. 
Observe that $f_{nm}^{-1}(z'_n)\cap f_{mk}(C_k)$ is a constructible subset of $X_m$ since 
$f_{mk}(C_k)$ is constructible by Chevalley's theorem and $f_{nm}^{-1}(z'_n)$ is closed since $z'_k$ is a closed point. 
Note also that $\varphi_{nm}^{-1}(z'_n)\cap \varphi_{mk}(Z_k)$ is the set of closed point of 
$f_{nm}^{-1}(z'_n)\cap f_{mk}(C_k)$ by Assertions (v) and (vi) of Lemma~\ref{l:mor-var-closed-points}. 
By applying again Lemma~\ref{l:nested-sequence}, we obtain
\[
\varphi_{nm}^{-1}(z'_n)=\bigcap_{m \leq k} (\varphi_{nm}^{-1}(z'_n)\cap \varphi_{mk}(Z_k)) \neq \varnothing.
\]
This shows that $\varphi_{n m}' = f_{nm}\vert_{Z_m'} \colon Z'_m \to Z'_n$ is surjective for all 
$n\leq m$. 
Thus $\varprojlim Z_n \neq \varnothing$.
\end{proof}

\bibliographystyle{siam}

\end{document}